\documentclass[a4paper, 12pt, headsepline]{scrartcl}
\setkomafont{disposition}{\normalfont\bfseries}

\pdfoutput=1

\usepackage[english, british]{babel}
\usepackage[utf8]{inputenc}
\usepackage[T1]{fontenc}
\usepackage{graphicx}
\usepackage{epstopdf}
\usepackage{pgfplots}

\usepackage{amsmath}
\usepackage{amsfonts}
\usepackage{amssymb}
\usepackage{amsthm}

\usepackage{caption}
\usepackage{subcaption}
\usepackage{pdfpages}
\usepackage[autostyle, english=british]{csquotes}

\usepackage[pdfborderstyle={/S/U/W 1}]{hyperref}
\hypersetup{pdfauthor={Caroline Lasser and David Sattlegger},pdftitle={Discretising the Herman--Kluk Propagator}}

\usepackage[
    backend=bibtex,
    bibencoding=ascii,
    style=alphabetic,
    natbib=true,
    url=false, 
    doi=true,
    eprint=true]{biblatex}
\AtEveryBibitem{\clearlist{language}} 
\AtEveryBibitem{\clearfield{note}} 
\AtEveryBibitem{\clearfield{issn}} 

\renewbibmacro{in:}{%
  \ifentrytype{article}{}{\printtext{\bibstring{in}\intitlepunct}}}

\def \bal#1\eal{\begin{align}#1\end{align}}
\def \bala#1\eala{\begin{align*}#1\end{align*}}
\def \expo#1{\exp \left( #1 \right)}

\providecommand{\pd}{\partial}
\def \d {\, \mathrm{d}}

\providecommand{\R}{\mathbb{R}}
\providecommand{\N}{\mathbb{N}}
\providecommand{\C}{\mathbb{C}}
\providecommand{\A}{\mathcal{A}}
\providecommand{\J}{\mathcal{J}}
\providecommand{\CC}{\mathcal{C}}

\def \L{{\mathrm{L}}}
\providecommand{\eps}{\varepsilon}
\providecommand{\pd}{\partial}
\def \i {\mathrm{i}}
\providecommand{\ii}{\mathrm{i}}
\def\Re{\mathrm{Re}}
\providecommand{\Heps}{\mathrm{H}^{\eps}}
\providecommand{\Uepst}{\mathrm{U}^{\eps}_t}
\providecommand{\Iepst}{\mathcal{I}^{\eps}_t}
\providecommand{\Ieps}{\mathcal{I}^{\eps}}

\providecommand{\LRd}{\L^2(\R^d)}
\providecommand{\D}{\mathcal{D}}
\providecommand{\W}{\mathcal{W}}

\providecommand{\SF}{\mathcal{S}}
\providecommand{\weyl}{\mathrm{op}^\eps}
\providecommand{\landauO}{\mathcal{O}} 
\providecommand{\Id}{\mathbf{I}}
\providecommand{\FT}{\mathcal{F}^\eps}
\providecommand{\e}{\mathrm{e}}
\providecommand{\supp}{\mathrm{supp}}
\providecommand{\Obs}{\mathcal{A}}
\providecommand{\E}{\mathbb{E}}
\providecommand{\Vb}{\mathbb{V}}
\def \geps #1 {g^{\eps}_{#1} }

\providecommand{\eg}{e.g.~}
\providecommand{\ie}{i.e.~}
\providecommand{\keywords}[1]{\textbf{Keywords}\; #1}
\providecommand{\subjclass}[1]{\textbf{2010 Mathematics Subject Classification}\; #1}

\def \norm#1{\left\| #1 \right\|}
\def \brac#1{\left( #1 \right)}
\def \abs#1{\left| #1 \right|}
\def \scal #1#2 { \left \langle #1 , #2 \right \rangle }

\theoremstyle{plain} 
\newtheorem{thm}{Theorem}
\newtheorem*{thm*}{Theorem}
\newtheorem{lemma}[thm]{Lemma} 
\newtheorem{proposition}[thm]{Proposition}

\theoremstyle{definition} 
\newtheorem{definition}{Definition}
\newtheorem*{definition*}{Definition}
\newtheorem{assumption}{Assumption}
\newtheorem{example}{Example}

\theoremstyle{remark} 
\newtheorem*{remark}{Remark}

\definecolor{myred}{rgb}{0.95,0,0}
\definecolor{mygreen}{rgb}{0,0.75,0}
\definecolor{myblue}{rgb}{0,0,0.85}

\title{Discretising the Herman--Kluk Propagator}

\author{ Caroline Lasser${}^*$ \hspace{1em} and \hspace{1em} David Sattlegger${}^\star$}
\date{\footnotesize \today}

\date{\normalsize 1${}^{\textrm{st}}$ May 2016}

\begin{document}
\maketitle

\begin{center}
\vspace{-22pt}
	Technische Universit\"at M\"unchen \\
	\medskip
	${}^*\!$ {\footnotesize \phantom{g}\href{mailto:classer@ma.tum.de}{classer@ma.tum.de}\phantom{g}} \\
	${}^\star\!$ {\footnotesize \href{mailto:sattlegg@ma.tum.de}{david.sattlegger@tum.de}}
\vspace{22pt}
\end{center}

\noindent

\bigskip\noindent

\begin{abstract}
The Herman--Kluk propagator is a popular semi-classical approximation of the unitary evolution operator in quantum molecular dynamics. In this paper we formulate the Herman--Kluk propagator as a phase space integral and discretise it by Monte Carlo and quasi-Monte Carlo quadrature. Then, we investigate the accuracy of a symplectic time discretisation by combining backward error analysis with Fourier integral operator calculus. Numerical experiments for two- and six-dimensional model systems support our theoretical results. 
\end{abstract}


\section{Introduction}

Molecular quantum dynamics is an active area of research aiming at an improved understanding of fundamental chemical processes, \eg photoisomerisation or electrochemical reactions. Calculations are based on the semi-classical Schr\"odinger operator
\bal 
\Heps=-\frac{\eps^2}{2} \Delta + V. 
\eal
which results from the Born--Oppenheimer approximation. Its potential $V: \R^d \to \R$ is a smooth function of sub-quadratic growth. The small positive parameter $\eps>0$ reflects the mass ratio of electrons and nuclei in a molecule and typically ranges between $10^{-3}$ and $10^{-2}$. Since $\Heps$ is a self-adjoint linear operator on $\L^2(\R^d)$, the spectral theorem provides a well-defined unitary propagator
\bal \label{eq:Uepst}
\Uepst = \e^{-\i\Heps t / \eps}
\eal
for all times $t\in\R$. This gives us existence and uniqueness of the solution 
\bal
\psi(t,\cdot) = \Uepst \psi_0
\eal
to the time-dependent Schrödinger equation
\bal\label{eqn:Schrodinger}
\i \ \eps \ \pd_t \psi = \Heps \psi,\qquad \psi(0,\cdot) = \psi_0,
\eal
for all square integrable initial data $\psi_0 \in \L^2(\R^d)$. 
Typical solutions to the time-dependent Schr\"odinger equation \eqref{eqn:Schrodinger} are wave packets with width of order~$\sqrt{\eps}$, wavelength of order $\eps$, and an envelope moving at velocity of order one. For small $\eps$, grid-based numerical methods need a very fine resolution and thus become expensive even in one and computationally infeasible in higher dimensions. In this situation, semi-classical methods come into play. They use the underlying classical Hamiltonian system 
\bala 
\dot z = \J \nabla h(z)
\eala
which is characterized by a Hamiltonian function $h:\R^{2d}\to\R$ and the matrix 
\bal
\J = \begin{pmatrix}0 & \Id_d \\ -\Id_d & 0\end{pmatrix} \in \R^{2d \times 2d}. 
\eal
Such a system is numerically accessible even in high dimensions. In addition these methods work with ansatz functions that have the correct localisation both in space and frequency, \eg a Gaussian wave packet
\bala
\geps{z} : \R^{2d}\to\R,\quad x \mapsto (\pi\eps)^{-d/4} \exp\brac{-\tfrac{1}{2\eps}\, |x-q|^2 +\tfrac{\i}{\eps} \, p\cdot(x-q)}. 
\eala
It is parametrised by a phase space point $z=(q,p)\in\R^{2d}$. Gaussian wave packets enjoy the striking property that any square integrable function $\psi\in \L^2(\R^d)$ can be decomposed according to 
\bala
\psi = \brac{2\pi\eps}^{-d} \int_{\R^{2d}} \scal{\geps{z}}{\psi} g^\eps_z \d z. 
\eala
The precise meaning of the integral is given by the inversion formula of the Fourier--Bros--Iagolnitzer (FBI) transform. From this we get the formal equation 
\bala
\Uepst \psi_0 = \brac{2 \pi \eps}^{-d} \int_{\R^{2d}} \scal{\geps{z}}{\psi_0 } \left(\Uepst \geps{z} \right) \d z 
\eala
which motivates the approximation of $\Uepst$ by continuously superimposing approximately propagated Gaussian wave packets. In the chemical literature such methods are known as Initial Value Representations, see \cite{Thoss2004}. From a mathematical viewpoint they constitute Fourier integral operators with complex valued phase functions. A very simple approximation, 
$$
U_t^\eps \geps{z} \;\approx\; \e^{\frac\i\eps S(t,z) }\geps{\Phi^t(z)} , 
$$ 
 is called Frozen Gaussian and is due to Heller \cite{Heller1981}. It evolves the wave packet's centre according to the classical flow 
\bala
\Phi^t : \R^{2d} \to \R^{2d}, \quad z \mapsto \begin{pmatrix} X^t(z) \\ \Xi^t(z) \end{pmatrix}, 
\eala
which is defined as the solution to the ordinary differential equation $\dot z = \J\nabla h(z)$ with initial datum $z(0) = z$. The phase of the wave packet changes according to the action integral along the classical trajectory, \ie
\bal\label{eq:action}
S(t,z) := \int_0^t \left( \tfrac{\d}{\d\tau} X^\tau(z) \cdot \Xi^\tau(z ) - h(\Phi^\tau(z))\right) \d \tau. 
\eal
The approximation by Herman and Kluk \cite{Herman1984} is more sophisticated as it accounts for the changes in the width of a wave packet. It is defined as an operator  \bal \notag
\Iepst : \L^2(\R^d) & \to \L^2(\R^d), \\ 
\psi & \mapsto \brac{2\pi\eps}^{-d} \int_{\R^{2d}} u(t,z) \, \e^{\frac{\i}{\eps}S(t,z)} \scal{\geps{z}}{\psi} \geps{\Phi^t(z) } \d z \label{eq:HK} 
\eal
and is nowadays called Herman--Kluk propagator. In its centre it encapsulates the well-known Herman--Kluk prefactor
\bala
u(t,z) := \sqrt{2^{-d} \det\!\left(\pd_q X^t(z) + \pd_p\Xi^t(z) + \i(\pd_q\Xi^t(z)-\pd_p X^t(z))\right)}
\eala
which depends on the components of the Jacobian matrix of the flow. \citet[Theorem~2]{Swart2009} prove that $\Iepst$ is a bounded operator on $\L^2(\R^d)$ and that it approximates the unitary propagator \eqref{eq:Uepst} in the following sense. For every $T>0$, there exists $C \ge 0$ such that for all $\eps>0$ 
\bal\label{eq:HK_accuracy}
\sup_{t \in [0,T]} \norm{ \Iepst - \Uepst} \leq C \, \eps.
\eal
The constant $C\ge0$ depends on higher order derivatives of the Hamiltonian function $h$ and the flow map $\Phi^t$. It vanishes, if the potential $V$ is a polynomial of degree $\le 2$, so that $\Iepst = U_t^\eps$ for all harmonic systems. 

The discretisation of the Herman--Kluk propagator involves two separate tasks, the phase space discretisation of the integral over $\R^{2d}$ and the time discretisation of the Hamiltonian flow together with the action and its Jacobian matrix. We present our approach to the phase space problem in \S\ref{section:PhaseSpaceDiscretization}. There, we assume that the initial data are Schwartz functions $\psi_0 \in \SF(\R^d)$ that allow for a multiplicative decomposition 
\bala
(2\pi\eps)^{-d} \scal{g^\eps_z}{\psi_0} =: r_0^\eps(z) \cdot \mu_0^\eps(z), 
\eala
for all $z \in \R^{2d}$, where $\mu_0^\eps\in\CC^\infty(\R^{2d})$ is a smooth probability distribution on $\R^{2d}$ and $r_0^\eps\in\CC^\infty(\R^{2d})\cap L^1(\d\mu_0)$ grows at most polynomially for $z \to \infty$. The Herman--Kluk propagator is thus rewritten as
\bala
\Iepst \psi_0 = \int_{\R^{2d}} r_0^\eps(z) \, u(t, z) \, \e^{\frac \i \eps S(t, z)} \, \geps{\Phi^t(z)} \d \mu^\eps_0(z).
\eala
We use Monte Carlo or quasi-Monte Carlo quadrature to discretise this integral by sampling $z_1, \ldots, z_M \in \R^{2d}$ from the probability 
distribution~$\mu_0$ and defining 
\bala
\psi_M^\eps(t) = \frac 1 M \sum_{m=1}^M r_0^\eps(z_m) \, u(t, z_m) \, \e^{\frac \i \eps S(t, z_m)} \, \geps{\Phi^t(z_m)} 
\eala
as a linear combination of Gaussian wave packets with classically propagated centres. We can prove that 
\bala
\psi_M^\eps(t) \to \Iepst\psi_0 \quad\text{as}\quad M\to\infty,
\eala
where the precise meaning of the limit and convergence rates will be addressed in \S~\ref{section:DiscretizingPhaseSpace}. For the time discretisation, which is presented in detail in \S\ref{sec:TimeDiscretisation}, we choose an initial phase space point $z\in\R^{2d}$ and set up a system of ordinary differential equations for 
\bala
\left(\Phi^t(z),(D\Phi^t)(z),S(t,z)\right)\in\R^{2d}\times\R^{2d\times 2d}\times\R.
\eala
We integrate the equations by a method of order $\gamma$ with fixed time step $\tau>0$ in such a way that we obtain a symplectic approximation $\tilde\Phi^\tau$ for the Hamiltonian flow $\Phi^t$. Denoting the corresponding approximate action and Herman--Kluk prefactor by $\tilde S$ and $\tilde u$ respectively we then define the approximate propagator $ \tilde{\mathcal{I}}^\eps_\tau : \L^2(\R^d)\to \L^2(\R^d)$ by 
\bala
\tilde{\mathcal{I}}^\eps_\tau \psi := (2\pi\eps)^{-d} \int_{\R^{2d}} \tilde u(\tau,z) \e^{\frac{\i}{\eps}\tilde S(\tau,z)} \langle g^\eps_z,\psi\rangle g^\eps_{\tilde\Phi^\tau(z)} \d z.
\eala
Our main theoretical result, Theorem~\ref{Thm:ConvergenceOfMethod}, establishes that the Herman--Kluk propagator and its time discrete counterpart are close to each other in the following sense. There exists a constant $C>0$ such that for all $\eps>0$ and $\tau>0$ with $\tau^{\gamma+1}<\eps$
\bala
\norm{\tilde{\mathcal{I}}^\eps_\tau - \mathcal I^\eps_\tau } \le C \tau^{\gamma+1}/\eps. 
\eala
For the proof we use the well-established backward error analysis of geometric numerical integration by \citet[Ch. IX]{Hairer2006} in combination with 
the Fourier integral operator calculus developed by \citet{Swart2009}. Our numerical experiments 
confirm the theoretical error estimate, of course, and demonstrate the practicability of the proposed discretisation in a moderately high-dimensional setting. All our simulations achieve an accuracy at the level of the asymptotic $O(\eps)$ resolution provided by \eqref{eq:HK_accuracy}.

The paper is organised as follows. The next section briefly reviews some numerical methods for the semi-classical Schr\"odinger equation. Then, we discuss the Herman--Kluk propagator and its properties in \S\ref{section:HK}. The algorithmic description of our discretisation is given in \S\ref{section:Algorithm}, while our main results, the convergence analysis of the phase space and time discretisation are presented in \S\ref{section:ApproximationProperties}. The numerical experiments in \S\ref{section:Experiments} comprise a two-dimensional torsional system as well as a Henon--Heiles system in dimension $d=6$. The appendices summarise computational details for the Herman--Kluk prefactor and expectation values. 


\section{Semi-classical approximations}\label{sec:semiclassics} 

The high frequencies of the solution to the semi-classical Schr\"odinger equation~\eqref{eqn:Schrodinger} 
exclude conventional grid based space discretisations schemes, in particular in view of the size of the 
dimension $d\gg1$ for molecular systems of interest. Quasi- and semi-classical approximations 
come into play here, using a priori analytical knowledge of the solution's qualitative behaviour. 
We will shortly review some of them. 


\subsection{Gaussian wave packets}

Gaussian wave packets $g^\eps_z$ are a major building block of the Herman--Kluk propagator~\eqref{eq:HK}. They are characterised by their respective centre point $z=(q,p)\in\R^{2d}$ in phase space. Their widths are frozen to be unit. Introducing a complex symmetric matrix $C=C^T\in\C^{d\times d}$ with positive definite imaginary part and a complex number $\xi\in\C$ to the parameter space, one defines a general Gaussian wave packet by
\bala
\varphi_0^\eps[z,C,\xi](x) := (\pi \eps)^{-d/4} \exp\brac{\tfrac \i{2 \eps} (x-q)\cdot C(x-q) + \tfrac \i \eps \, p\cdot (x-q) + \tfrac\i\eps \, \xi}. 
\eala
Note that this definitions contains the simple Gaussian wave packet 
\bala
g^\eps_z = \varphi_0^\eps[z,\i\,\Id_d,0].
\eala
as a special case. If the phase and normalisation parameter $\xi$ is properly chosen with respect to the width matrix $C$, then
\bala
\norm{\varphi_0^\eps[z,C,\xi]}^2 = \int_{\R^d} |\varphi^\eps_0[z,C,\xi](x)|^2 \d x = 1.
\eala 
For the unitary propagation of a general Gaussian wave packet, one supplements the Hamiltonian system $\dot z(t) = \J\nabla h(z(t))$ for the centre motion by a Riccati equation for the complex width matrix $C(t)$ and an ordinary differential equation for $\xi(t)$ ensuring the correct phase and normalisation. Then, for every $T>0$, there is a constant $c\ge 0$ such that for all $\eps>0$
\bal 
\sup_{t\in[0,T]} \norm{\varphi_0^\eps[z_t,C_t,\xi_t] - \Uepst\varphi^\eps_0[z_0,C_0,\xi_0] } \le c \sqrt{\eps}.
\eal
Moreover, if the potential $V$ is a polynomial of degree $\le 2$, then $c=0$, and the Gaussian wave packet approximation is exact. Over decades, general Gaussian wave packets have been used as a flexible tool in chemical physics, cf.~Heller \cite{Heller1976} or Littlejohn \cite{Littlejohn1986}. More recently, they have also been considered for the systematic construction of numerical integrators by Faou and Lubich \cite{Faou2006}. 


\subsection{Hagedorn's semi-classical wave packets} \label{section:Hagedorn} 

Any complex symmetric matrix $C=C^T\in\C^{d\times d}$ with positive definite imaginary part can be written as $C=PQ^{-1}$, where $P,Q\in\C^{d\times d}$ are invertible and satisfy
\bal \label{eqn:symprel}
Q^T P - P^T Q = 0, \qquad Q^* P - P^* Q = 2 \i\, \Id_d. 
\eal
We use such matrices $P$ and $Q$ to build the rectangular matrix
$$
Z = \begin{pmatrix}Q\\ P\end{pmatrix}\in\C^{2d\times d}.
$$ 
Then, we define the general Gaussian wave packet
\bala
\lefteqn{\varphi^\eps_0[z,Z](x) =}\\
&& (\pi \eps)^{-d/4} \det(Q)^{-1/2} \exp\left(\tfrac \i{2 \eps} (x-q)\cdot PQ^{-1}(x-q) + \tfrac \i \eps p\cdot (x-q)\right)
\eala
in the parametrisation introduced by Hagedorn \cite{Hagedorn:SemclassicalQM:1980,Hagedorn1998}. The matrix conditions~\eqref{eqn:symprel} ensure the correct normalisation, 
$$
\norm{\varphi^\eps_0[z,Z]} = 1.
$$
Hagedorn's parametrisation allows an elegant construction of an orthonormal basis of $L^2(\R^d)$ 
$$
\varphi_k^\eps[z,Z] = \frac{1}{\sqrt{k!}}\A^\dagger[z,Z]^k\varphi_0^\eps[z,Z],\qquad k\in\N^d,
$$
by the iterated application of the raising operator
$$
\A^\dagger[z,Z] = \frac{\i}{\sqrt{2\eps}} \,Z^* \J (\hat z - z), \text{ where } (\hat z\psi)(x) := \begin{pmatrix}x\psi(x)\\ -\i\eps\nabla\psi(x)\end{pmatrix}.
$$
For the unitary propagation of these semi-classical wave packets one augments the Hamiltonian equation $\dot z_t = \J\nabla h(z_t)$ by a rectangular version of its variational equation
$$
\dot Z_t = \J \nabla^2h(z_t) Z_t
$$
and the action integral \eqref{eq:action} to generalise the previously discussed Gaussian wave packet approximation as shown by Hagedorn \cite[Theorem~2.9]{Hagedorn1998}. For all $k\in\N^d$ and $T>0$ there exists $c\ge0$ such that for all $\eps>0$
\bala
\sup_{t\in[0,T]}\norm{\Uepst \varphi_k^\eps [z_0,Z_0] - \e^{\frac \i \eps S(t,z_0)} \varphi_k^\eps [z_t, Z_t] } \leq c \sqrt{\eps}.
\eala
Again, if the potential $V$ is a polynomial of degree $\le 2$, then $c=0$. Using this exact propagation property for harmonic Hamiltonians, \citet{Faou2009} as well as  \citet{Gradinaru2014} develop a Galerkin method with time-splitting for a convergent discretisation of the unitary time evolution of Hagedorn's semi-classical wave packets. 


\subsection{Gaussian beams}

A complementary line of semi-classical approximations is built for initial data that are less localised in position space than semi-classical wave packets. Wentzel--Kramers--Brillouin (WKB) wave functions 
\bala
w^\eps_0(x) = \alpha_0(x) \e^{\frac\i\eps \sigma_0(x)},\qquad x\in\R^d,
\eala
are defined by a complex-valued amplitude function $\alpha_0\in C^\infty(\R^d)$ and a real-valued phase function $\sigma_0\in C^\infty(\R^d)$. A first order Gaussian beam approximation of the unitary Schr\"odinger dynamics carries WKB initial data beyond caustics by continuously superimposing general Gaussian wave packets according to
\bal\label{eq:beam}
b^\eps_t= (2\pi\eps)^{-d/2} \int_{\Lambda_0} \alpha_t(z)\, \e^{\frac\i\eps \sigma_t(z)} \,\varphi^\eps_0[\Phi^t(z),Z_t(z)] \d z.
\eal
The centres of the initial Gaussians are chosen from the set 
\bala
\Lambda_0 := \left\{(x,\nabla \sigma_0(x))\mid x\in\supp(a_0)\right\},
\eala
while the propagation of the beam parameters $\alpha_t(z)\in\C$, $\sigma_t(z)\in\R$, and $Z_t(z)\in\C^{2d\times d}$ is achieved by a system of coupled ordinary differential equations driven by the classical Hamiltonian flow $\Phi^t:\R^{2d}\to\R^{2d}$. Its building blocks resemble the variational equation and the equation for the action integral. Zhen \cite[Theorem~5.1]{Zheng2014} proves that for all $T>0$ there exists a constant $c\ge0$ such that for all $\eps>0$
$$
\sup_{t\in[0,T]} \norm{\Uepst w^\eps_0 - b^\eps_t} \le c \,\eps.
$$
Higher order Gaussian beam approximations with $\landauO(\eps^{N/2})$ accuracy, $N\in\N$, have been developed as well \cite{Runborg2013}. The discretisation of the continuous Gaussian beam superposition \eqref{eq:beam} and its higher order versions has been tackled by grid based numerical quadrature. Thus, numerical applications have been restricted to systems in dimension $d=1$ and $d=2$.


\subsection{Quasi-classical approximations}\label{sec:quasi}

It is often not the time-evolved wave function $\psi(t,\cdot)=\Uepst\psi_0$ which is of interest, but derived quadratic quantities as expectation values
$$
E_\Obs(t) = \langle \psi(t),\Obs \psi(t)\rangle
$$
for a given self-adjoint operator $\Obs$ defined on $L^2(\R^d)$. Typical observables are $\eps$-scaled pseudo-differential operators and can be expressed as the Weyl quantisation $\Obs = \weyl(a)$ of a smooth phase space function $a:\R^{2d}\to\R$. Consider for example 
$$
\psi\mapsto-\frac{\eps^2}{2}\Delta\psi\quad\text{and}\quad \psi\mapsto V\psi,
$$ 
\ie the kinetic and potential energy operators respectively. The most popular quasi-classical approximation \cite{Miller1974,Thoss2004,Lasser2010} uses the Wigner function $\W(\psi_0)\in L^2(\R^{2d})$ of the initial wave function and the classical Hamiltonian flow $\Phi^t:\R^{2d}\to\R^{2d}$ to compute the weighted phase space integral
$$
E_{\mathrm{appr}}(t) = \int_{\R^{2d}} (a \circ \Phi^t)(z) \, \W(\psi_0)(z) \, \d z.
$$
This is commonly called linearised semi-classical initial value representation (LSC-IVR) in chemistry journals. Its accuracy is of order two in $\eps$, meaning that for all $T>0$ there exists a constant $c\ge0$ such that for all $\eps>0$
\bal 
\sup_{t\in[0,T]} \left| E_\Obs(t) - E_{\mathrm{appr}}(t)\right| \le c\,\eps^2.
\eal
The constant $c$ depends on the observable $\Obs$ and derivatives of the flow $\Phi^t$, but is uniformly bounded for all normalised initial data with $\norm{\psi_0}=1$. 
As for the Hagedorn wave packets and the Herman--Kluk propagator, the time evolution for quadratic Hamiltonians is exact so that $c=0$ in this case. In \S\ref{section:ExamplesExpectationValues} we shall use this quasi-classical approximation to calculate reference expectation values for our numerical experiments in $d=6$.


\section{The Herman--Kluk propagator} \label{section:HK}

In \cite{Herman1984} Herman and Kluk observed that in most cases a single Gaussian wave packet cannot accurately approximate a quantum system. However, a suitable superposition of Gaussian wave packets can. The authors provide a formal justification and derive what we now call the Herman--Kluk propagator. The rigorous mathematical analysis of this method is due to \cite{Swart2009}. It crucially uses the following generalised Fourier transform.
\begin{definition} \label{def:FBI}
For $z=(q,p)\in\R^{2d}$ we set
\bal \label{eqn:CoherentState}
\geps{z} (x) = (\pi \eps)^{- \frac d 4} \exp\left( -\frac{1}{2 \eps} \abs{x-q}^2 + \frac{\i}{\eps} \ p \cdot \brac{x-q} \right),\qquad x\in\R^d.
\eal
The mapping $T^\eps: \SF(\R^d) \longrightarrow \SF(\R^{2d})$ defined by 
\bala
& (T^\eps \psi)(z) := (2 \pi \eps)^{- d / 2} \scal{\geps{z}}{\psi}
\eala
is called the Fourier--Bros--Iagolnitzer (in short: FBI) transform. 
\end{definition}

One can show that the FBI transform can be extended to map $\LRd$ isometrically into $\L^2(\R^{2d})$ and that for all $\psi \in \LRd$ the inversion formula
\bal \label{eqn:FBIInversion}
\psi = \brac{2 \pi \eps}^{-d} \int_{\R^{2d}} g^{\eps}_{z} \, \scal{\geps{z}}{\psi } \d z
\eal
holds, see \cite[Chapter 3.1]{Martinez2002}. From this we get the formal equation 
\bala
\Uepst \psi_0 = \brac{2 \pi \eps}^{-d} \int_{\R^{2d}} \left(\Uepst \geps{z} \right) \, \scal{\geps{z}}{\psi_0 } \d z 
\eala
which is used as a starting point for the Herman--Kluk approximation. 


\subsection{Definition and well-posedness}
\begin{definition}
For any initial wave function $\psi_0 \in \LRd$ and time $t \in [0, T]$ the \textit{Herman--Kluk propagator} is defined by
\bal \label{eqn:HK}
\Iepst \psi_0 = \brac{2 \pi \eps}^{-d} \int_{\R^{2d}} u(t, z) \, \e^{\frac \i \eps S(t, z)} \geps{\Phi^t(z)} \, \scal{\geps{z}}{\psi_0} \d z . 
\eal 
Again, $\Phi^t=(X^t,\Xi^t)$ denotes the classical Hamiltonian flow and $S$ the 
corresponding action 
\bala
S(t, \cdot ) = \int_0^t \left( \tfrac{\d}{\d\tau} X^\tau \cdot \Xi^\tau - h(\Phi^\tau)\right) \d \tau. 
\eala
The quantity $u(t, z) $ is called \textit{Herman--Kluk prefactor}. It incorporates the components of the Jacobian matrix of the flow
\bala
(D\Phi^t)(z) = \begin{pmatrix}\pd_q X^t(z) & \pd_p X^t(z)\\ \pd_q \Xi^t(z) & \pd_p \Xi^t(z)\end{pmatrix}\in\R^{2d\times 2d}
\eala 
and is defined by 
\bal \label{eqn:HKFactor}
u(t, z) := \sqrt{2^{- d} \det \brac{ \pd_q X^t(z) - \i \pd_p X^t(z) + \i \pd_q \Xi^t(z) + \pd_p \Xi^t(z) } }
\eal
 for all $z = (q,p) \in \R^{2d}$. 
\end{definition} 

\begin{remark}
Note that for $t=0$ the Herman--Kluk propagator reduces to the FBI inversion formula \eqref{eqn:FBIInversion}, that is, 
\bala
\Ieps_0\psi_0 = \psi_0,\qquad \psi_0\in \LRd,
\eala
since $\Phi^0(z) = z$, $u(0,z) = 1$, and $S(0,z) = 0$. 
\end{remark}
\citet{Swart2009} introduce a general class of Fourier integral operators to which the Herman--Kluk propagator $\Iepst$ belongs and prove that it is possible to construct approximate propagators of arbitrary order in $\eps$. The following theorem is a special case of their main result. 

\begin{thm}[{Swart and Rousse \cite[Theorem 2]{Swart2009}}] \label{Thm:TheoreticBound}
Let $\Uepst$ be the unitary time evolution of \eqref{eqn:Schrodinger} with subquadratic potential $V$. The Herman--Kluk propagator $\Iepst$ satisfies 
\bala
\sup_{t \in [0,T]} \left \| \Uepst- \mathcal{I}^{\eps}_t \right\|_{\LRd \rightarrow \LRd} \leq C(T) \ \eps, 
\eala
where $T>0$ is a fixed time and $C(T)>0$ is independent of $\eps$. 
\end{thm} 
In light of this approximation estimate we desire a numerically stable Herman--Kluk algorithm. This is the main contribution of our paper. In the process of proving our main result in \S\ref{section:ApproximationProperties} we shall also use elements of the Fourier integral operator calculus that has been developed in \cite{Swart2009} for establishing Theorem~\ref{Thm:TheoreticBound}.

\begin{remark}
As an intermediate result of the original proof in \cite{Swart2009}, one obtains that for any time $t\in[0,T]$ the prefactor $z\mapsto u(t,z)$ is a smooth function such that the function itself and all its derivatives are bounded. Moreover, one also discovers that the Herman--Kluk propagator is exact for quadratic Hamiltonians. 
\end{remark}


\subsection{The Herman--Kluk propagator in momentum space} 
Many situations require knowledge of the Fourier transform of a wave function, \eg when calculating the expectation values for the momentum operator $\psi\mapsto\-\i\eps\nabla\psi$ or the kinetic energy operator $\psi\mapsto-\frac{\eps^2}{2}\Delta\psi$. Since in general we will not have the Herman--Kluk wave function on a uniform grid, using the FFT might prove difficult. There is, however, a way to calculate the Herman--Kluk propagator and its Fourier transform simultaneously by considering the following formal argument. For all $\xi\in\R^d$ let
\bala
\brac{\FT\psi}(\xi) := (2\pi\eps)^{-d/2} \int_{\R^d} \e^{-\frac\i\eps x\cdot\xi}\,\psi(x) \d x
\eala
be the $\eps$-scaled Fourier transform. Then, 
\bala
\FT\brac{\Iepst \psi_0} &= \brac{2 \pi \eps}^{-d} \ \FT \brac{ \int_{\R^{2d}}u(t,z) \, \e^{\frac \i \eps S(t,z)} \geps{\Phi^t(z)} \scal{\geps{z}}{\psi_0} \d z } \\ 
 &= \brac{2 \pi \eps}^{-d} \int_{\R^{2d}} u(t,z) \, \e^{\frac \i \eps S(t, z)} \brac{\FT \geps{\Phi^t(z)} } \, \scal{\geps{z}}{\psi_0} \d z . 
\eala
Once one manages to calculate the Herman--Kluk propagator, it is sufficient to know the Fourier transform of a Gaussian wave packet, \ie
$$
\FT \geps{(q,p)} = \e^{-\frac\i\eps p\cdot q}\, g^\eps_{(p,-q)}, 
$$ 
to calculate its Fourier transform. This can be done in parallel without substantial additional effort. 


\section{The algorithm} \label{section:Algorithm}

As the first step for deriving the algorithm that we propose, let us take another look at the definition of the Herman--Kluk propagator \eqref{eqn:HK}. Its evaluation requires involves one integral over the phase space $\R^{2d}$ and another one over $\R^d$ for each phase space point in order to calculate the FBI transform. In either case the respective integrand is potentially highly oscillatory. Furthermore, we need to calculate the classical flow $\Phi^t(z)$, the classical action $S(t,z)$, and the Herman--Kluk prefactor $u(t,z)$ for all phase space points $z\in\R^{2d}$. The present chapter describes a way to do this while circumventing any difficulties that occur along the way. 


\subsection{Phase space discretisation} \label{section:PhaseSpaceDiscretization}
In order not to having to evaluate the integral for the FBI transform by numerical quadrature, we restrict ourselves to specific initial wave functions. Their FBI transform should be computable analytically and they should satisfy the following assumption.

\begin{assumption} \label{AssumptionMuR}
Let $\psi_0 \in \SF(\R^d)$ such that for all $z \in \R^{2d}$ there is a multiplicative decomposition 
\bal \label{eq:alpha}
(2\pi\eps)^{-d} \scal{g^\eps_z}{\psi_0} =: r^\eps_0(z) \cdot \mu^\eps_0(z), 
\eal
with $\mu^\eps_0\in \SF(\R^{2d})$ being a probability distribution on $\R^{2d}$ and the complex-valued function $r^\eps_0\in C^\infty(\R^{2d})\cap L^1(\d\mu^\eps_0)$ growing at most polynomially for $z \to \infty$. 
\end{assumption}
A variety of initial wave functions that are commonly used in semi-classical calculations satisfy this assumption, including Hermite functions as well as Hagedorn wave packets. 
\begin{example}[label=FBIofGaussian]\label{Ex:PhaseSpaceDisc}
A common choice as initial wave function is a simple Gaussian wave packet $\psi_0=g^\eps_{z_0}$ centred at some point $z_0 = (q_0, p_0) \in \R^{2d}$. In this case, the scalar product that occurs in the FBI transform gives 
\bal \label{eqn:FBIofGauss}
 \scal{\geps{z}}{\psi_0} = \exp\brac{-\tfrac{1}{4 \eps} \, \abs{z - z_0}^2 + \tfrac{\i}{2 \eps} \, \brac{p + p_0} \cdot \brac{q-q_0} }. 
\eal
Hence, we get 
\bala
\mu_0^\eps(z) = (4 \pi \eps)^{-d} \e^{-\frac{1}{4 \eps} \abs{z - z_0}^2}\quad\text{and}\quad
r_0^\eps(z) = 2^{d} \e^{\tfrac \i {2 \eps} \brac{p + p_0} \cdot \brac{q - q_0} }.
\eala
as a multiplicative decomposition thereof. For the corresponding explicit formulae for Hermite and Hagedorn functions see \cite{Lasser2014}. 
\end{example}
Assumption~\ref{AssumptionMuR} allows the interpretation of the Herman--Kluk propagator as an integration over phase space weighted 
with respect to the probability measure $\mu_0$, 
\bal \label{eqn:WeightedIntegration}
\Iepst \psi_0 = \int_{\R^{2d}} r^\eps_0(z) \, u(t, z) \, \e^{\frac \i \eps S(t, z)} \, \geps{\Phi^t(z)} \d \mu^\eps_0(z). 
\eal
For one-dimensional problems we could consider grid based quadrature methods for the $\mu_0^\eps$-integration. However, already for two-dimensional systems phase space is four-dimensional, and conventional grid based approaches are no longer practical. We therefore turn to grid free methods, in particular Monte Carlo and quasi-Monte Carlo quadrature, which permit the evaluation of high dimensional integrals. In addition, their shortcoming of having a low order of accuracy is of little consequence since the total error is already dominated by the asymptotic error of order $\eps$ as shown in Theorem~\ref{Thm:TheoreticBound}. We use either Monte Carlo or quasi-Monte Carlo quadrature to approximate \eqref{eqn:WeightedIntegration}. In both cases we define an approximate wave function by 
\bal \label{eqn:ApproxWaveFunction}
\psi_M(t) := \frac 1 M \sum_{m=1}^M r_0^\eps(z_m) \, u(t, z_m) \, \e^{\frac \i \eps S(t, z_m)} \, \geps{\Phi^t(z_m)} 
\eal
where $z_1, \ldots, z_M \in \R^{2d}$ are sampled from $\mu_0^\eps$. In \S\ref{section:DiscretizingPhaseSpace} we will present rigorous error estimates for these discretisations.


\subsection{Calculation of expectation values} \label{section:ExpectationValues}
One of the  Herman--Kluk propagator's advantages is the ability to compute the full wave function including its phase. In addition, we also want to be able to calculate expectation values for observables. This is important for practical purposes as well as comparability to reference solutions. A quantum mechanical observable is a self-adjoint operator $\Obs$ on $\L^2(\R^d)$, \eg the position or momentum operator. Its expectation value with respect to a normalised state $\psi \in \L^2(\R^d)$ is given by the inner product $\scal{\psi}{\Obs \, \psi} $. In order to calculate such quantities we would have to perform yet another numerical quadrature with an highly oscillatory integrand. However, there is a way to compute expectation values without actually evaluating the full Herman--Kluk wave function. By using the abbreviation $f_t^\eps(z) := r_0(z) \, u(t, z) \, \e^{\frac \i \eps S(t, z)}$ we write
\bal
\scal{\Iepst \psi_0}{\Obs \, \Iepst \psi_0} &= \int_{\R^{2d}} \int_{\R^{2d}} \overline{f_t^\eps(w)} \, f_t^\eps(z) \, \scal{\geps{\Phi^t(w)}}{\Obs \, \geps{\Phi^t(z)} } \d \mu_0^\eps(w) \d \mu_0^\eps(z) \nonumber \\ \label{eqn:ApproxExpectationValues}
& = \int_{\R^{4d}} \overline{f_t^\eps(w)} \, f_t^\eps(z) \, \scal{\geps{\Phi^t(w)}}{\Obs \, \geps{\Phi^t(z)} } \d \brac{\mu_0^\eps \otimes \mu_0^\eps}(w, z) . 
\eal
This way we interpret the expectation value as a weighted integral on $\R^{4d}$ with respect to the product measure $\mu_0^\eps \otimes \mu_0^\eps$ instead of two separate integrations on $\R^{2d}$. If we consider a sequence of (Monte Carlo or quasi-Monte Carlo) quadrature points 
\bala 
(w_1, z_1),\ldots,(w_M, z_M) \in \R^{4d} 
\eala 
that are sampled from $\mu_0^\eps \otimes \mu_0^\eps$, then 
\bal \label{eqn:ApproxExpValue}
A_M(t) := \frac{1}{M} \sum_{m=1}^M \overline{f_t^\eps(w)} \, f_t^\eps(z) \scal{\geps{\Phi^t(w^m)}}{\Obs \, \geps{\Phi^t(z^m)} }
\eal
is an approximation to \eqref{eqn:ApproxExpectationValues}. Note that the computational effort grows linearly in the number of quadrature points albeit on a space of twice the dimension. In addition, we may even find analytic expressions for
\bal
\scal{\geps{\Phi^t(w)}}{\Obs \, \geps{\Phi^t(z)}}
\eal
for several observables including position, momentum, and kinetic energy operators, as well as all polynomial potentials and the torsional potential. Some examples are given in Appendix \ref{appendix:ExpectiationValues}. 


\subsection{Time discretisation} \label{sec:TimeDiscretisation}
In order to preserve the symplectic structure of the classical Hamiltonian system
\bal\label{eqn:HamiltonianSystem}
\dot z = \J \, \nabla h(z), 
\eal
we need a suitable numerical integrator. In addition to the flow of the Hamiltonian system we have to compute the Herman--Kluk factor $u(t,z)$ and the classical action $S(t,z)$. The computation of $u(t,z)$ requires the solution to the variational equation 
\bal \label{eqn:Variational}
\dot W(t) = \J \, \nabla^2 h(\Phi^t) \, W(t), \quad W(0) = \Id_{2d},
\eal
where $W(t) = D_z \Phi^t$ is the derivative of the flow with respect to the initial values and $\nabla^2 h$ is the Hessian of the Hamiltonian function. For a separable system of the form $h(q,p) = T(p) + V(q)$ the classical action may be seen as solution to the initial value problem 
\bal
\dot S(t, z) = T(\Xi^t(z)) - V(X^t(z)), \quad S(0, z) = 0
\eal
for all $z=(q,p) \in \R^{2d}$. Let us artificially spilt this equation into two, defining $S_T$ and $S_V$ by 
\bal \label{eqn:Action}
\begin{pmatrix} \dot{S}_T(t,z) \\ \dot{S}_V(t,z) \end{pmatrix} = \begin{pmatrix} T(\Xi^t(z)) \\ - V(X^t(z)) \end{pmatrix}, %
\quad \begin{pmatrix} S_T(0, z) \\ S_V(0, z) \end{pmatrix} = \begin{pmatrix} 0 \\ 0 \end{pmatrix} 
\eal
Then we may solve \eqref{eqn:HamiltonianSystem}, \eqref{eqn:Variational}, and \eqref{eqn:Action} simultaneously by a single numerical integrator. In our numerical experiments we use a composition method based on the St\o rmer--Verlet scheme which is symplectic and symmetric, cf. \cite[Chapter~VI]{Hairer2006}. The order of the scheme is controlled by using a composition strategy with composition constants taken from \cite{Kahan1997}. If we assume a separable system of the form $h(q,p) = T(p) + V(q)$ the resulting method is an explicit one, which makes our calculations even more efficient. 


\subsection{Schematic description of the algorithm}

Our goal is to calculate either a wave function, more precisely the solution to the Schr\"odinger equation, or expectation values of operators along this solution. The two tasks require different sampling points but may use the same time-step algorithm. 
\begin{enumerate}
\item
	\begin{enumerate}
	\item Sample $z_1, \ldots, z_M \in \R^{2d}$ from $\mu_0^\eps$; 
	\item[] \textbf{or}
	\item Sample $(w_1, z_1),\ldots,(w_M, z_M) \in \R^{4d}$ from $\mu_0^\eps \otimes \mu_0^\eps$; 
	\end{enumerate}
\item Allocate an array containing the sampling points and the corresponding initial values for the variational equations and the classical action; 
\item Evolve this array according to \eqref{eqn:HamiltonianSystem}, \eqref{eqn:Variational}, and \eqref{eqn:Action} using a high-order symplectic and symmetric numerical integration method; 
\item Compute the Herman--Kluk factor with a continuous phase (cf. Appendix \ref{appendix:ComputationHKFactor}) and the action respectively. 
\item 
	\begin{enumerate}
	\item Calculate the approximate wave function by formula \eqref{eqn:ApproxWaveFunction}; 
	\item[] \textbf{or}
	\item Calculate expectation values by means of formula \eqref{eqn:ApproxExpValue};
	\end{enumerate}
\end{enumerate}
Because of their parallel nature, these algorithms can be implemented in a highly efficient manner. A related article is currently in preparation.


\section{Approximation properties of the algorithm} \label{section:ApproximationProperties} 

The previous section proposes an algorithm for the computation of the \mbox{Herman--Kluk} propagator. Two quantities have to be discretised. The first one is an integral over phase space, the second one a solution to a system of ordinary differential equations. We continues with a systematic analysis of the errors that result from these two discretisation steps. 


\subsection{Phase space discretisation} \label{section:DiscretizingPhaseSpace}
Let us first discretise the phase space integral. In order to facilitate notation we denote the integrand by 
\bal \label{eqn:MCIntegrand}
f^\eps_{t}(z) := r_0^\eps(z) \, u(t, z) \, \e^{\frac \i \eps S(t, z)} \geps{\Phi^t(z)} \in\SF(\R^d)
\eal
with $z \in \R^{2d}$ and $t\in [0,T]$, where $\mu_0^\eps$ and $r_0^\eps$ are chosen as in Assumption \ref{AssumptionMuR}. Then, 
\bala
\Iepst \psi_0 = \int_{\R^{2d}} f^\eps_{t}(z) \d \mu_0^\eps(z)  \in \LRd. 
\eala


\subsubsection{Using Monte Carlo quadrature} \label{sec:MC}
For Monte Carlo quadrature we treat the integrand $f^\eps_{t}$ as a random variable with values in the Hilbert space $\LRd$ distributed according to the probability measure $\mu_0^\eps$ and interpret the phase space integral as its expected value, \ie 
\bal
\Iepst \psi_0 = \int_{\R^{2d}} f^\eps_{t}(z) \d \mu_0^\eps(z) = \E[f^\eps_{t}]. 
\eal
By taking $M$ independent samples $z_1, \dots, z_M\in\R^{2d}$ of the probability distribution~$\mu_0^\eps$ we then define the Monte Carlo estimator 
\bal \label{eqn:MCEstimator}
\psi_M^\eps(t) := \frac 1 M \sum_{m=1}^M f^\eps_{t}(z_m). 
\eal
Note that this is just a linear combination of classically evolved Gaussian wave packets. We obtain the following estimate for its mean squared error, which shows the usual $\landauO(M^{-1/2})$ behaviour with respect to the number of sample points.
\begin{proposition}\label{prop}
Let the initial wave function $\psi_0\in\SF(\R^d)$ satisfy Assumption~\ref{AssumptionMuR} and consider the Monte Carlo estimator $\psi_M^\eps(t)$ defined in \eqref{eqn:MCEstimator}. Then, the mean squared error is given by 
\bala
\E\! \left[ \norm{\psi_M^\eps(t) - \Ieps_t\psi_0}^2 \right] =  \frac{\Vb[f_t^\eps]}{M},
\eala
where $\Vb[f_t^\eps]$ satisfies 
\bala
\Vb[f_t^\eps] \le 3 \int_{\R^{2d}} |u(t,z) r_0^\eps(z)|^2 \d\mu_0^\eps(z) + \norm{\Iepst\psi_0}^2 
\eala
for all $t\in[0,T]$ and $\eps>0$. 

\begin{proof}
We observe that 
\bal \label{eqn:UnbiasedEstimator}
\E[\psi_M^\eps(t)] = \frac 1 M \sum_{M=1}^M \E[f_{t}^\eps] = \Iepst \psi_0.
\eal
Since the samples are independent and identically distributed, we get
\bala
\E\! \left[ \norm{\psi_M^\eps(t) - \Ieps_t\psi_0}^2 \right] = \Vb[\psi_M^\eps(t)] = \frac{1}{M^2} \sum_{m=1}^M \Vb[f_t^\eps] = \frac{\Vb[f_t^\eps]}{M}. 
\eala
Moreover,
\bala
\Vb[f_t^\eps] 
&= \E\!\left[\norm{f_t^\eps - \Ieps_t\psi_0}^2\right] = \int_{\R^{2d}} \norm{f_t^\eps(z)-\Ieps_t\psi_0}^2 \d\mu_0^\eps(z)\\
&= \int_{\R^{2d}} \norm{f_t^\eps(z)}^2 \d\mu_0^\eps(z) - 2 \int_{\R^{2d}} \Re\langle f_t^\eps(z),\Ieps_t\psi_0\rangle \d\mu_0(z) + \norm{\Ieps_t\psi_0}^2.
\eala
By writing 
\bala
\int_{\R^{2d}} \langle f_t^\eps(z),\Ieps_t\psi_0\rangle \d\mu_0^\eps(z) = 
\int_{\R^{4d}} \langle f_t^\eps(z),f_t^\eps(w)\rangle \d(\mu_0^\eps\otimes\mu_0^\eps)(w,z)
\eala
and estimating  
\bala
\abs{\langle f_t^\eps(z),f_t^\eps(w)\rangle}\le \norm{f_t^\eps(z)}\norm{f_t^\eps(w)} \le \frac12(\norm{f_t^\eps(z)}^2 + \norm{f_t^\eps(w)}^2).
\eala 
we therefore find that  
\bala
\abs{\int_{\R^{2d}} \langle f_t^\eps(z),\Ieps_t\psi_0\rangle \d\mu_0^\eps(z)} \le \int_{\R^{2d}} \norm{f_t^\eps(z)}^2 \d\mu_0^\eps(z).
\eala
Since $\norm{f_t^\eps(z)} = \abs{u(t,z)r_0^\eps(z)}$, we conclude the estimate as
\bala
\Vb[f_t^\eps] \le 3 \int_{\R^{2d}} \abs{u(t,z)r_0^\eps(z)}^2 \d\mu_0^\eps(z) + \norm{\Ieps_t\psi_0}^2.
\eala
\end{proof}
\end{proposition}

The final estimate of Proposition~\ref{prop},
\bala
\Vb[f_t^\eps] \le 3 \int_{\R^{2d}} |u(t,z) r_0^\eps(z)|^2 \d\mu_0^\eps(z) + \norm{\Iepst\psi_0}^2,
\eala
is dominated by its first summand, since Theorem~\ref{Thm:TheoreticBound} provides
\bala
\norm{\Ieps_t\psi_0} = \norm{\Uepst\psi_0} + \landauO(\eps) = \norm{\psi_0} + \landauO(\eps).
\eala
In the case of our previous example we may even calculate the initial variance~$\Vb[f_0^\eps]$ analytically and observe $\eps$-independence as well as convergence to one as $d\to\infty$.

\begin{example}[continues=FBIofGaussian] \label{Ex:MonteCarlo}
For the initial mean squared error of the sampling of a simple Gaussian wave packet $\psi_0 = \geps{z_0} $ we can compute an analytic expression for the variance. We have 
\bala
f_{0}^\eps(z) = r_0^\eps(z) \geps{z} = 2^d \,\e^{\tfrac \i {2 \eps} \brac{p + p_0} \cdot \brac{q - q_0} } \geps{z} \quad\text{and}\quad \E[f_{0}^\eps] = \geps{z_0}
\eala
so that
\bala
\Vb[f_0^\eps] = \int_{\R^{2d}} \abs{r_0^\eps(z)\geps{z} (x)-\geps{z_0} (x)}^2 \d x \d\mu_0^\eps(z) = 1 - 4^{-d}.
\eala
This expression will be underlined by the numerical experiments in \S\ref{sec:SamplingInitialWF}. 
\end{example}


\subsubsection{Using quasi-Monte Carlo quadrature}

Quasi-Monte Carlo quadrature is an equiweighted quadrature on well-chosen deterministic quadrature points. Let $z_1,\ldots,z_M\in\R^{2d}$ and denote by
\bala
\D_M^\eps(z_1,\ldots,z_M;z) = \frac{1}{M} \sum_{m=1}^M \chi_{]-\infty, z]}(z_m) -\mu_0^\eps(]-\infty, z]),\qquad z\in\R^{2d},
\eala
the discrepancy function of the probability measure $\mu_0^\eps$ that quantifies the deviation of the empirical distribution for the rectangular interval
\bala
]-\infty,z] \, := \, ]-\infty,z_1]\times\cdots\times]-\infty,z_M]\subset\R^{2d}.
\eala
If the measure $\mu^\eps_0$ is the product of one-dimensional probability measures so that the inverses of the one-dimensional cumulative distribution functions are accessible, then the well-established low discrepancy sets for the uniform measure on the unit cube $[0,1]^{2d}$ allow to construct points $z_1,\ldots,z_M\in\R^{2d}$ with 
$$
\sup_{z\in\R^{2d}}\abs{\D_M^\eps(z_1,\ldots,z_M;z)} = \landauO\!\left((\log M)^{2d-1} / M\right),
$$
see \cite[Theorem~4]{Aistleitner2015}. The following lemma elucidates, why the discrepancy function is crucial for equiweighted quadrature. 

\begin{lemma}\label{lem:Koksma} Let $f\in\SF(\R^{2d})$ and $\mu_0^\eps$ be a probability distribution on $\R^{2d}$ so that $f\in L^1(\d\mu^\eps_0)$. Then, for all $z_1,\ldots,z_M\in\R^{2d}$
\bala
\frac{1}{M} \sum_{m=1}^M f(z_m) - \int_{\R^{2d}} f(z) \d\mu_0^\eps(z) = \int_{\R^{2d}} \pd^{1:2d} f(z) \,\D_M^\eps(z_1,\ldots,z_M;z)\d z,
\eala
where $\pd^{1:2d} = \pd_1\pd_2\cdots\pd_{2d}$ denotes the mixed partial derivative through all dimensions.
\end{lemma}

We shall present the proof of Lemma~\ref{lem:Koksma} in Appendix~\ref{appendix:QuasiMonteCarlo} and now turn to its application for the phase space discretisation of the Herman--Kluk propagator. 
We consider $z_1,\ldots,z_M\in\R^{2d}$ and set 
\bal\label{eq:QMC}
\psi_M^\eps(t) := \frac 1 M \sum_{m=1}^M f^\eps_{t}(z_m) 
\eal
with the function
\bala
f_{t}^\eps(z) := r_0^\eps(z) \, u(t, z) \, \e^{\frac \i \eps S(t, z)} \geps{\Phi^t(z)} \in\SF(\R^d)
\eala 
for $z\in\R^{2d}$ and $t\in[0,T]$. We obtain the following weak convergence result. 

\begin{proposition} \label{thm:ConvergenceQMC}
Let $\psi_0 \in \SF(\R^d)$ and $r_0^\eps, \mu_0^\eps$ be defined according to Assumption~\ref{AssumptionMuR} and consider $\psi^\eps_M(t)$ as defined in \eqref{eq:QMC}. Then,
\bal\label{eq:QMC_result}
\psi_M^\eps(t)-\Iepst\psi_0 = \int_{\R^{2d}} \pd^{1:2d}_z f^\eps_t(z) \,\D^\eps_M(z_1,\ldots,z_M;z) \d z.
\eal
In particular, $\lim_{M\to\infty} \sup_{z\in\R^{2d}}\abs{\D^\eps_M(z_1,\ldots,z_M;z)} = 0$ implies for all test functions $\phi\in\SF(\R^d)$
\bal\label{eq:convergence}
\lim_{M \to \infty} \left\langle\phi,\psi_M^\eps(t) - \Iepst \psi_0\right\rangle = 0 .
\eal
\end{proposition}

\begin{proof}
We observe that for all $\phi\in\SF(\R^d)$ the mapping $z\mapsto \langle\phi,f^\eps_t(z)\rangle$ defines a Schwartz function on $\R^{2d}$. We therefore apply Lemma~\ref{lem:Koksma} to obtain
\bala
\langle\phi,\psi_M^\eps(t)-\Iepst\psi_0\rangle = \int_{\R^{2d}} \pd^{1:2d}_z \langle\phi,f^\eps_t(z)\rangle \,\D^\eps_M(z_1,\ldots,z_M;z) \d z,
\eala
which means
\bala
\psi_M^\eps(t)-\Iepst\psi_0 = \int_{\R^{2d}} \pd^{1:2d}_z f^\eps_t(z) \,\D^\eps_M(z_1,\ldots,z_M;z) \d z.
\eala
Moreover, 
\bala
\abs{\left\langle\phi,\psi_M(t) - \Iepst \psi_0\right\rangle} \le \sup_{z\in\R^{2d}}\abs{\D_M^\eps(z_1,\ldots,z_M;z)} 
\int_{\R^{2d}} \abs{\pd^{1:2d}_z\langle\phi,f^\eps_t(z)\rangle} \d z,
\eala
so that $\lim_{M\to\infty} \sup_{z\in\R^{2d}}\abs{\D^\eps_M(z_1,\ldots,z_M;z)} = 0$ implies \eqref{eq:convergence}.
\end{proof}

Even though we have proven weak convergence, we notice that the mixed derivative of our integrand $f^\eps_t(z)$ depends unfavourably on various parameters as our next example illustrates. 

\begin{example}[continues=Ex:MonteCarlo] \label{Ex:QMC}
We examine the mixed derivative of the initial integrand $f^\eps_0(z)$ for a Gaussian wave packet $\psi_0 = \geps{0} $ centred in the origin $z_0 = 0$. 
We calculate
\bala
\pd^{1:2d} f^\eps_0(z) = f^\eps_0(z) \prod_{j=1}^d \frac{\i}{2\eps^{2}} \left( (x-q_j-\i p_j)(x-\tfrac12 q_j)\right) 
\eala
and obtain 
\bala
\norm{\pd^{1:2d} f^\eps_0(z)}^2 
= \eps^{-4d} \,\prod_{j=1}^d \left(\tfrac18\eps(q_j^2 + 6\eps) + \tfrac14 p_j^2(q_j^2 + 2\eps) \right).
\eala
for the square of the norm. Hence, the norm of the mixed derivative has a multiplicative factor $\eps^{-2d}$ in front of a polynomial in $z$. Our numerical experiments in \S\ref{section:Experiments} indeed confirm that the smaller $\eps$ and the larger the dimension $d$, the more quadrature points are required. However, it seems that beneficial cancellations in the key equation \eqref{eq:QMC_result} allow for a much smaller $M$ than expected. 
\end{example}


\subsection{Error due to the ode solver}

In Theorem \ref{Thm:TheoreticBound} we learned that the Herman--Kluk propagator $\Iepst$ approximates the unitary time evolution $\Uepst$ in the sense
\bala
\sup_{t\in[0,T]} \norm{\Uepst - \Iepst} \leq C(T) \, \eps . 
\eala
Let us examine the time discretised Herman--Kluk propagator 
\bala
\tilde \Iepst: \LRd \to \LRd 
\eala
which is defined by
\bala
\tilde\Iepst\psi := (2\pi\eps)^{-d} \int_{\R^{2d}} \tilde u(t,z)\, \e^{\frac{\i}{\eps}\tilde S(t,z)} \langle \geps_z ,\psi\rangle \geps_{\tilde \Phi^t(z)} \d z
\eala
It depends on the flow $\tilde\Phi^t$, the action $\tilde S$, and the prefactor $\tilde u$ that are computed by the symplectic numerical integrator proposed in \S\ref{sec:TimeDiscretisation}. The following theorem is the main result of this paper. It relates the local accuracy of the time discrete Herman--Kluk propagator with the one of the ode discretisation. 

\begin{thm} \label{Thm:ConvergenceOfMethod}
Let $\gamma$ be the order of the symplectic integrator of the algorithm in \S\ref{sec:TimeDiscretisation}. 
There exists a constant $C > 0$ such that the time discrete Herman--Kluk propagator satisfies
\bala
\norm{ \tilde{\mathcal I}^\eps_\tau -{\mathcal I}^\eps_\tau} \ \leq \ C\, \tau^{\gamma+1}/\eps.
\eala 
for all $\eps>0$ and all time steps $\tau>0$ with $\tau^{\gamma+1} < \eps$. 
\end{thm}
In order to prove Theorem~\ref{Thm:ConvergenceOfMethod} we combine backward error analysis of symplectic integrators with the calculus of Fourier integral operators. Let us thus review the basic concepts of these two fields. 
\subsubsection{Backward error analysis}
We summarise the basic ideas of backward error analysis as presented in \cite[Chapter~IX]{Hairer2006}. We need to solve a Hamiltonian system
\bala
\dot z = \J\nabla h(z)
\eala
with flow map $\Phi^t:\R^{2d}\to\R^{2d}$. If we compare this to the flow $\tilde\Phi^\tau$ of a symplectic numerical discretisation of order~$\gamma$ with time step $\tau>0$, we find that 
\bala
\tilde\Phi^\tau(z) = \Phi^\tau(z) + \landauO(\tau^{\gamma+1}).
\eala
Furthermore, $\tilde\Phi^\tau$ is the exact flow to a modified Hamiltonian system
\bala
\dot z = \J \nabla \tilde h(z)
\eala
with Hamiltonian 
\bala
\tilde h(z) = h(z) + \landauO(\tau^\gamma),
\eala
as shown in \cite[\S IX.8]{Hairer2006}. The Herman--Kluk prefactor is built from the Jacobian matrix of the flow map, so that the discretised prefactor $\tilde u(\tau,z)$ inherits its local accuracy, 
\bal\label{eq:udis}
\tilde u(\tau,z) = u(\tau,z) + \landauO(\tau^{\gamma+1}).
\eal
For the action integral, we obtain the same property via the following lemma. 
\begin{lemma}\label{lem:act}
The action integral $S(t,z)$ of the flow map $\Phi^t$ and its time discrete counterpart $\tilde S(\tau,z)$ satisfy
\bala
\tilde S(\tau,z) = S(\tau,z) + \landauO(\tau^{\gamma+1}).
\eala
\end{lemma}

\begin{proof}
Let us split the difference of the two action integrals into four parts. 
\bala
&\tilde S(\tau,z) - S(\tau,z) = \\
& \phantom{ - } \int_0^\tau \tfrac{\d}{\d s}\left(\tilde X^{s}(z)-X^s(z)\right)\cdot\tilde \Xi^s(z) \d s &+  \int_0^\tau \tfrac{\d}{\d s}X^s(z)\cdot \left( \tilde \Xi^s(z)-\Xi^s(z)\right)\d s & \\
& - \int_0^\tau \left(\tilde h(\tilde\Phi^s(z))-\tilde h(\Phi^s(z))\right) \d s &- \int_0^\tau \left(\tilde h(\Phi^s(z))-h(\Phi^s(z))\right) \d s . \quad \, &
\eala
Each of the four integrands is at most $\landauO(\tau^\gamma)$, so that integration over the interval $[0,\tau]$ results in $\landauO(\tau^{\gamma+1})$.
\end{proof}


\subsubsection{Fourier integral operators}

The class of Fourier integral operators considered by \citet{Swart2009} comprises the Herman--Kluk propagator as a special case. Let $\Phi^t$ be a smooth Hamiltonian 
flow and $S$ the associated action. If 
\bala
u:\R\times\R^{2d}\times\R\to\C,\quad (t,z,x)\mapsto u(t,z,x)
\eala
is a smooth function with bounded derivatives, then 
\bala
\mathcal I(\Phi^t,u)\psi(x) := (2\pi\eps)^{-d} \int_{\R^{2d}} u(t,z,x) \e^{\frac{\i}{\eps}S(t,z)} \left\langle \geps_z ,\psi\right\rangle \geps_{\Phi^t(z)} (x) \d z
\eala
defines a bounded operator on $\LRd$. According to \citet[Theorem~1]{Swart2009}, whenever $(t,z)\mapsto u(t,z)$ is a smooth function that does not depend on $x$, then one can estimate the operator norm as
\bal\label{eq:linf}
\norm{\mathcal I(\Phi^t,u)} \le 2^{-d/2} \norm{u(t,\cdot)}_\infty.
\eal
Moreover, particular $x$-dependent linear factors absorb an inverse power of the semi-classical parameter $\eps$. That is, by \cite[Lemma~3]{Swart2009}, we have 
\bala
\mathcal I(\Phi^t,\tfrac{1}{\eps}(x_j-X^t_j)u) = \mathcal I(\Phi^t,v)
\eala
for all $j=1,\ldots,d$, where
\bal\label{eq:absorb}
v(t,z,x) := -\mathrm{div}_z\left(e_j\cdot \mathcal Z_t^{-1}(z) u(t,z,x)\right).
\eal
Here, $e_j\in\C^d$ is the $j$th standard basis vector, and 
\bala
\mathrm{div}_z f := \sum_{k=1}^d \pd_{q_k} f_k - \i \sum_{k=1}^d \pd_{p_k} f_k
\eala
for smooth vector valued functions $f:\R^{2d}\to\C^d$. 
Furthermore, we define a smooth mapping to the set of invertible complex $d\times d$ matrices 
\bala
\mathcal Z_t := \pd_q X^t - \i \pd_p X^t + \i \pd_q \Xi^t + \pd_p \Xi^t
\eala 
using the four blocks of the Jacobian matrix of the flow map $\Phi^t$. 


\subsubsection{Derivatives of Gaussian wave packets}

The last building block of the proof of Theorem~\ref{Thm:ConvergenceOfMethod} is the calculation of 
the derivatives of a Gaussian wave packet with respect to its phase space centre. For the gradient, we obtain
\bala
\nabla_z \geps_z (x) = 
\begin{pmatrix}\frac{1}{\eps}(x-q)-\tfrac{\i}{\eps}p\\*[1ex] \frac{\i}{\eps}(x-q)\end{pmatrix} \geps_z (x)
=: v^\eps(x-q,p) \geps_z (x),
\eala
where $z = (q,p)\in\R^{2d}$ and $x\in\R^d$. The higher order derivatives can be expressed in terms of products of multivariate polynomials with the Gaussian wave packet.

\begin{lemma}\label{lem:gauss} 
For fixed $x\in\R^d$, we consider the function $\R^{2d}\to\C$, $z\mapsto \geps_z (x)$. Then, for any multi-index $\alpha\in\N^{2d}$, there exists a multi-variate polynomial 
$\mathcal P^\eps_\alpha:\R^{2d}\to\C$ 
of degree $|\alpha|$ such that for all $z=(q,p)\in\R^{2d}$,
\bala
D^\alpha_z \geps_z (x) = \mathcal P_\alpha^\eps(x-q,p) \,\geps_z (x).
\eala
In particular, 
\bala
\mathcal P^\eps_\alpha(x-q,p) = \sum_{k+|\beta|\le |\alpha|} \lambda_{\alpha}(k,\beta)\, \eps^{-k} \,v^\eps(x-q,p)^\beta,
\eala
where the coefficients $\lambda_\alpha(k,\beta)\in\C$ are $\eps$-independent complex numbers indexed by $(k,\beta)\in\N\times\N^{2d}$. 
\end{lemma}

\begin{proof}
We argue by induction and calculate 
\bala
D^{\alpha+e_j}_z \geps_z (x) &= 
\pd_j \left( \mathcal P^\eps_\alpha(x-q,p) \geps_z (x)\right)\\*[1ex]
&=
\left( \mp(\pd_j \mathcal P^\eps_\alpha)(x-q,p) + \mathcal P^\eps_\alpha(x-q,p) \,e_j\cdot v^\eps(x-q,p) \right)\geps_z (x),
\eala
where the $\mp$ sign depends on whether $j \in \{1,\ldots,d\}$ or $j \in \{d+1,\ldots,2d\}$. Finally, we observe that
\bala
& \mp(\pd_j \mathcal P^\eps_\alpha)(x-q,p) + \mathcal P^\eps_\alpha(x-q,p) \,e_j\cdot v^\eps(x-q,p)\\*[1ex]
& \quad = 
\sum_{k+|\beta|\le|\alpha|} \lambda_\alpha(k,\beta)\, \eps^{-k}\left( \mp \pd_j v^\eps(x-q,p)^\beta + v^\eps(x-q,p)^{\beta+e_j}\right) \\
& \quad = 
\sum_{k+|\beta|\le|\alpha|+1} \lambda_{\alpha+e_j}(k,\beta)\, \eps^{-k}\,v^\eps(x-q,p)^\beta.
\eala
\end{proof}


\subsubsection{The proof of Theorem~\ref{Thm:ConvergenceOfMethod}}
In the last three paragraphs we prepared everything we need in order to for prove Theorem~\ref{Thm:ConvergenceOfMethod}.
\begin{proof}
We estimate the accuracy of the time discrete Herman--Kluk propagator in four steps.

\paragraph{\normalfont\em Towards the first estimate.}
We write
\bala
\tilde {\mathcal I}^\eps_\tau - {\mathcal I}^\eps_\tau &= {\mathcal I}(\tilde\Phi^\tau,\tilde u)-{\mathcal I}(\Phi^\tau,u)\\
& = \mathcal I(\tilde\Phi^t,\tilde u- u)+ \mathcal I(\tilde\Phi^\tau,u)-\mathcal I(\Phi^\tau,u), 
\eala
so that \eqref{eq:linf} and \eqref{eq:udis} imply
\bala
\norm{\mathcal I(\tilde\Phi^t,\tilde u- u)} \ \le\ 2^{-d/2} \norm{\tilde u(\tau,\cdot) - u(\tau,\cdot)}_\infty = \landauO(\tau^{\gamma+1}).
\eala

\paragraph{\normalfont\em Towards the second estimate.}
Hence, for the rest of the proof we are only concerned with
\bala
& \mathcal I(\tilde\Phi^\tau,u)-\mathcal I(\Phi^\tau,u) \\
& = (2\pi\eps)^{-d} \int_{\R^{2d}} u(\tau,z)
 \left( \e^{\frac{\i}{\eps}\tilde S(\tau,z)} \geps_{\tilde\Phi^\tau(z)} - \e^{\frac{\i}{\eps}S(\tau,z)} \geps_{\Phi^\tau(z)} \right) 
 \left\langle \geps_z ,\cdot \right\rangle \d z.
\eala
We express the difference in the integrand as
\bala
& \e^{\frac{\i}{\eps}\tilde S(\tau,z)} \geps_{\tilde\Phi^\tau(z)} - \e^{\frac{\i}{\eps}S(\tau,z)} \geps_{\Phi^\tau(z)} \\
& =
\left( 1 - \e^{\frac{\i}{\eps}(S(\tau,z)-\tilde S(\tau,z))}\right) \e^{\frac{\i}{\eps}\tilde S(\tau,z)} \geps_{\tilde\Phi^\tau(z)} + 
\e^{\frac{\i}{\eps}S(\tau,z)} \left(\geps_{\tilde\Phi^\tau(z)} - \geps_{\Phi^\tau(z)} \right)
\eala
and denote 
\bala
w_0(\tau,z) := 1 - \exp\!\left(\tfrac{\i}{\eps}\left(S(\tau,z)-\tilde S(\tau,z)\right)\right).
\eala 
Then, 
\bala
& \mathcal I(\tilde\Phi^\tau,u)-\mathcal I(\Phi^\tau,u) = \mathcal I(\tilde\Phi^\tau, u w_0) \\
&\qquad + 
(2\pi\eps)^{-d} \int_{\R^{2d}} u(\tau,z)
\e^{\frac{\i}{\eps}S(\tau,z)} \left(\geps_{\tilde\Phi^\tau(z)} - \geps_{\Phi^\tau(z)} \right)\left\langle \geps_z ,\cdot \right\rangle \d z,
\eala
and Lemma~\ref{lem:act} yields
\bala
\norm{\mathcal I(\tilde\Phi^\tau, u w_0) } \le 2^{-d/2}\norm{u(\tau,\cdot)w_0(\tau,\cdot)}_\infty = \landauO(\tau^{\gamma+1}/\eps).
\eala

\paragraph{\normalfont\em Towards the third estimate.}
We now use the Taylor series of the analytic function $z\mapsto \geps_z (x)$ around a point $z_0\in\R^{2d}$, 
\bala
\geps_z (x) - \geps_{z_0} (x) = (z-z_0)\cdot \nabla_z \geps_{z} (x)\mid_{z=z_0} + \sum_{|\alpha|\ge 2} \frac{(z-z_0)^\alpha}{\alpha!} \, D^\alpha_z \geps_z (x) \mid_{z=z_0}.
\eala
We denote 
\bala
\delta^\tau(z) := (D\Phi^\tau)(z) \left(\tilde\Phi^\tau(z)-\Phi^\tau(z)\right),
\eala
observe that 
\bala
\delta^\tau(z) = \landauO(\tau^{\gamma+1}),
\eala 
and use Lemma~\ref{lem:gauss} to write 
\bala
\geps_{\tilde\Phi^\tau(z)} (x) - \geps_{\Phi^\tau(z)} (x) = \left( w_{1,1}^\eps(\tau,z,x) + w_{1,2}^\eps(\tau,z) + r^\eps(\tau,z,x)\right) \geps_{\Phi^\tau(z)} (x)
\eala
with
\bala
w_{1,1}^\eps(\tau,z,x) &= \delta^\tau(z) \cdot \begin{pmatrix}\frac{1}{\eps}(x-X^t(z)) \\*[1ex] \frac{\i}{\eps} (x-X^t(z))\end{pmatrix} \\
&= \frac{1}{\eps}\sum_{j=1}^d \left(\delta_j^\tau(z) +\i\delta_{j+d}^\tau(z)\right) \left(x_j-X_j^t(z)\right),\\*[1ex]
w_{1,2}^\eps(\tau,z) &= \delta^\tau(z) \cdot \begin{pmatrix}-\frac{\i}{\eps}\Xi^t(z) \\ 0\end{pmatrix},\\*[1ex]
r^\eps(\tau,z,x) &= \sum_{|\alpha|\ge 2} \frac{ \delta^\tau(z)^\alpha}{\alpha !} \,\mathcal P^\eps_\alpha(x-X^\tau(z),\Xi^\tau(z)) .
\eala
This implies that
\bala
& (2\pi\eps)^{-d} \int_{\R^{2d}} u(\tau,z)
\e^{\frac{\i}{\eps}S(\tau,z)} \left(\geps_{\tilde\Phi^\tau(z)} - \geps_{\Phi^\tau(z)} \right)\left\langle \geps_z ,\cdot \right\rangle \d z\\
& = 
\mathcal I(\Phi^\tau, uw_{1,1}^\eps) + \mathcal I(\Phi^\tau,uw_{1,2}^\eps) + \mathcal I(\Phi^\tau,u r^\eps)
\eala
By \eqref{eq:absorb}, we have 
\bala
\mathcal I(\Phi^\tau,uw^\eps_{1,1}) = \mathcal I(\Phi^\tau,v_1)
\eala
with
\bala
v_1(\tau,z) = - \sum_{j=1}^d {\mathrm{div}}_z\left( e_j\cdot \mathcal Z_\tau^{-1}(z) u(\tau,z) \left(\delta_j^\tau(z) +\i\delta_{j+d}^\tau(z)\right)\right).
\eala
Applying \eqref{eq:linf} once more, we have
\bala
\norm{\mathcal I(\Phi^\tau,uw^\eps_{1,1})} & \le 2^{-d/2} \norm{v_1(\tau,\cdot)}_\infty = \landauO(\tau^{\gamma+1}),\\*[1ex]
\norm{\mathcal I(\Phi^\tau,uw^\eps_{1,2})} & \le 2^{-d/2}\norm{u(\tau,\cdot) w^\eps_{1,2}(\tau,\cdot)}_\infty = \landauO(\tau^{\gamma+1}/\eps).
\eala

\paragraph{\normalfont\em Towards the fourth estimate.}
It remains to bound $\mathcal I(\Phi^\tau,ur^\eps)$. By Lemma~\ref{lem:gauss}, 
\bala
\mathcal P^\eps_\alpha(x-X^\tau(z),\Xi^\tau(z)) = \sum_{(k,\beta)\in M_\alpha} \lambda_\alpha(k,\beta) \,\eps^{-k}\, v^\eps(x-X^\tau(z),\Xi^\tau(z))^\beta.
\eala
Therefore, the crucial terms in $r^\eps$ are of the form
\bala
\delta^\tau(z)^\alpha \,\eps^{-k}\, v^\eps(x-X^\tau(z),\Xi^\tau(z))^\beta
\eala
with $|\alpha|\ge 2$ and $k+|\beta|\le |\alpha|$. 
The previous arguments for bounding $\mathcal I(\Phi^\tau,uw^\eps_{1,1})$ and 
$\mathcal I(\Phi^\tau,uw^\eps_{1,2})$ then provide
\bala
\norm{\mathcal I(\Phi^\tau,ur^\eps)} \ =\ \landauO((\tau^{\gamma+1}/\eps)^{|\alpha|}) \ = \ \landauO(\tau^{\gamma+1}/\eps).
\eala
\end{proof}


\section{Numerical examples} \label{section:Experiments}
Let us underline the results of the previous section with a series of numerical examples. First, we will test the robustness of our algorithm by calculating the full wave function of the quantum mechanical harmonic oscillator problem in one dimension and comparing it to the analytic solution. Next, we will do the same for the torsional potential in $2d$ using a reference solution that is computed by a split-step Fourier method. After that, we calculate expectation values using the approach presented in \S\ref{section:ExpectationValues}. We shall do this again for a $2d$ torsional potential and - in order to underline the capability for calculation high-dimensional problems - the Henon--Heiles potential in $6d$. Finally, we illustrate one of our main results, Theorem \ref{Thm:ConvergenceOfMethod}, by examining the behaviour of the overall error of our method with respect to the time step size of the underlying symplectic ode solver. 


\subsection{Approximation of the initial wave function}\label{sec:SamplingInitialWF}

We first examine the quality of our algorithm with respect to the discretization of phase space as described in \S\ref{section:PhaseSpaceDiscretization}. Let us continue with our example from \S\ref{sec:MC}. 

\begin{example}[continues=Ex:MonteCarlo]
For the sampling of the initial Gaussian wave function $\psi_0=\geps{z_0} $ we found that for
\bala
f_{0}^\eps(z) := 2^d \e^{\tfrac \i {2 \eps} \brac{p + p_0} \cdot \brac{q - q_0} } \geps{z} . 
\eala
we have 
\bala
\E[f^\eps_{0}] = \geps{z_0} \quad\text{and}\quad\Vb(f^\eps_0) = 1 - 4^{-d}. 
\eala
Figure \ref{MCInitialSampling} shows the sampling error for the initial wave function with respect to the number of Monte Carlo quadrature points $M$. Each wave function is produced by averaging over $10$ independent samples. The two pictures show the error for one and two space dimensions respectively. Note that the error shows no dependence on the value of $\eps$. 
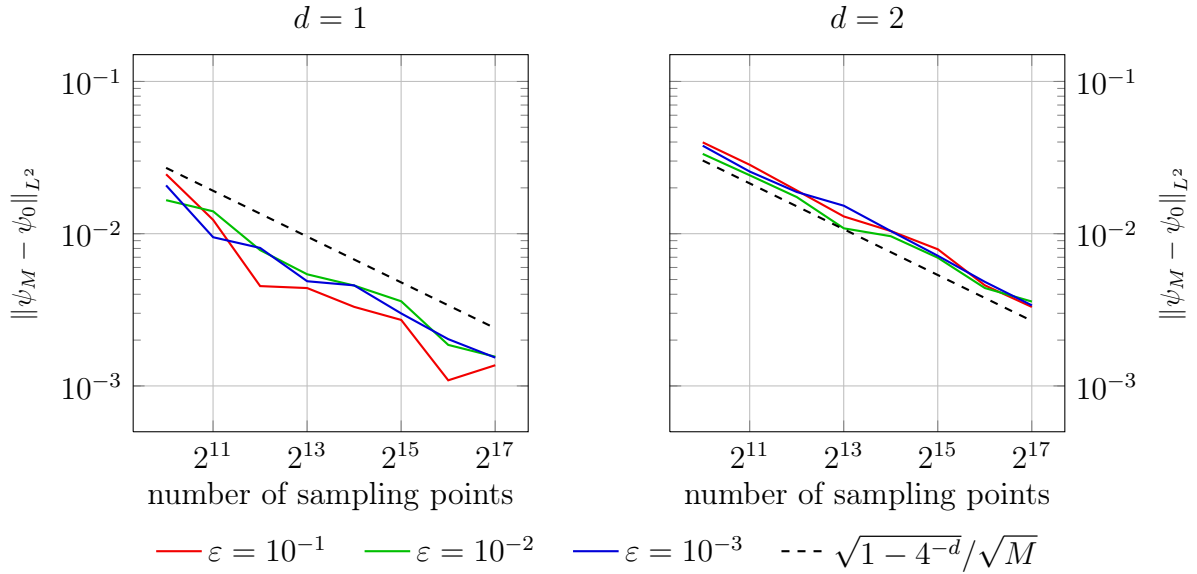
\begin{figure}[!ht] 
\begin{centering}
 	 	 \begin{tikzpicture} 
 	 	 	 \begin{loglogaxis}[ 	
			title = {$d=1$}, 	
			title style={align=left, font=\normalsize},
			width = 0.33 \textwidth, 	
			height = 5cm, 		
			scale only axis, 
 		log basis x={2},
			xtickten = {11,13,15,17},
			ytickten={-3,-2,-1},
			ymin = 0.0005,
			ymax = 0.15, 
			ylabel near ticks, yticklabel pos=left, 
			xlabel={number of sampling points}, 
			ylabel={$ {\norm{\psi_M - \psi_0 }}_{L^2} $}, 
			ylabel style={font=\footnotesize}, 
			grid=major,
			legend columns = -1, 
			legend style={draw = none}, 
	 		legend to name= LegendMCInitialSampling, 
			legend entries={$\eps = 10^{-1}\quad$,$\eps = 10^{-2}\quad$,$\eps = 10^{-3}\quad$,$\sqrt{1 - 4^{-d}}/\sqrt M$}] 
			\addplot[thick,myred] table[x index=0,y index=2]{plotdata/InitialSampling_d=1_Nx=512.dat}; 
			\addplot[thick,mygreen] table[x index=0,y index=3]{plotdata/InitialSampling_d=1_Nx=512.dat}; 
			\addplot[thick,myblue] table[x index=0,y index=4]{plotdata/InitialSampling_d=1_Nx=512.dat}; 
			\addplot[thick,black,dashed] table[x index=0,y index=1]{plotdata/InitialSampling_d=1_Nx=512.dat}; 
	 	 	 \end{loglogaxis} 
 	 	 \end{tikzpicture} 
\hfill
 	 	 \begin{tikzpicture} 
 	 	 	 \begin{loglogaxis}[ 	
			title = {$d=2$}, 	
			title style={align=left, font=\normalsize},
			width = 0.33 \textwidth, 	
			height = 5cm, 		
			scale only axis, 
 		log basis x={2},
			xtickten = {11,13,15,17},
			ytickten={-3,-2,-1},
			ymin = 0.0005,
			ymax = 0.15, 
			ylabel near ticks, yticklabel pos=right, 
			xlabel={number of sampling points}, 
			ylabel={$ {\norm{\psi_M - \psi_0 }}_{L^2} $}, 
			ylabel style={font=\footnotesize}, 
			grid=major] 
			\addplot[thick,myred] table[x index=0,y index=2]{plotdata/InitialSampling_d=2_Nx=256.dat}; 
			\addplot[thick,mygreen] table[x index=0,y index=3]{plotdata/InitialSampling_d=2_Nx=256.dat}; 
			\addplot[thick,myblue] table[x index=0,y index=4]{plotdata/InitialSampling_d=2_Nx=256.dat}; 
			\addplot[thick,black,dashed] table[x index=0,y index=1]{plotdata/InitialSampling_d=2_Nx=256.dat}; 
	 	 	 \end{loglogaxis} 
 	 	 \end{tikzpicture} 
\ref*{LegendMCInitialSampling}
\caption{Initial sampling error for $\psi_0=\geps{z_0} $ in dimensions $d=1$ (upper panel) and $d=2$ (lower panel) with respect to the number of Monte Carlo quadrature points $M$.} 
\label{MCInitialSampling} 
\end{centering}
\end{figure} 

\end{example}

\subsection{Time evolution of the wave function}
Now we shall use both discretisations, i.e. in time and phase space, to calculate the solution to the semi-classical Schr\"odinger equation for different potentials. 
\begin{example}[The harmonic oscillator] \label{section:Harmonic Oscillator}
The quantum mechanical harmonic oscillator is one of the few examples for which an analytic solution is known explicitly. Furthermore, the Herman--Kluk propagator is exact for quadratic potentials. As a proof of concept we will restrict ourselves to one dimension where a grid based approach is still feasible. This allows us to test and demonstrate the robustness of our algorithm even for large times, in this case $t \in [0,100]$. Let us consider the harmonic oscillator potential $V(x) = x^2/2$ and initial data 
\bala
\psi_0 = \geps{z_0} .
\eala
Let $q(t)$, $p(t)$ and $S(t)$ be the position, momentum and action of the classical harmonic oscillator, i.e. 
\bala
q(t) &= x_0 \cos(t) + \xi_0 \sin(t) \\
p(t) &=\xi_0 \cos(t) - x_0 \sin(t) \\
S(t) &= \tfrac{1}{2} \sin(t) \left(\left(\xi_0^2-x_0^2\right) \cos(t)-2 \xi_0 x_0 \sin(t)\right). 
\eala 
Then the analytic solution to the quantum mechanical problem is given by 
\bala
\psi_a(t,x) = (\pi \eps)^{-1/4} \exp\brac{ \frac{\i}{\eps} \, S(t) - \frac{\i}{2} \, t - \frac{1}{2 \eps} \brac{x-q(t)}^2 + \frac{\i}{\eps} \, p(t) (x-q(t)) }, 
\eala
cf. \cite[Thm 2.5]{Hagedorn1998}. 

For the numerical calculations consider an equidistant grid in classical phase space $\R^2$ with grid size $\delta_q = \delta_p = 0.1$. Consider another equidistant grid in the wave function's position space with grid size $\delta_x = 2^{-8} \pi$ on the interval $[-\pi, \pi]$. The time is discretised in equally spaced steps $\tau = 0.05 $, starting at $t=0$ up to the final time $t=100$. As initial position and momentum we take $x_0 = 1$ and $\xi_0 = 0$. 
Figure \ref{fig:ErrorHKtoAnalyticHarmonic1d} shows the error between the Herman--Kluk and the analytic solution in the $\L^2$-norm for different values of the semi-classical parameter $\eps \in \{10^{-1},10^{-2},10^{-3}\}$. It underlines that the Herman--Kluk propagator is exact for quadratic potentials and that our algorithm preserves this feature even over long times. 
\begin{figure}[!ht]
	\begin{center}
		\begin{tikzpicture}
			\begin{semilogyaxis}[
				scale only axis, 
				height=5cm,
				width=0.75 \textwidth,
				xlabel={time},
				ylabel near ticks, yticklabel pos=left, 
				ylabel={$\norm{\psi-\psi_a}_{\L^2}$},
				xmin = -2,
				xmax = 102,
				grid=major,
				legend columns = -1, 
				legend style={cells={anchor=west},
				legend pos= south east}]
				\addplot[very thick, myred] table[x index=0,y index=1]{plotdata/error2analytic_HK_d=1_potential=harmonic_multiple_epsilons.dat};
				\addplot[very thick, dashed, myred] table[x index=0,y index=2]{plotdata/error2analytic_HK_d=1_potential=harmonic_multiple_epsilons.dat};
				\addplot[very thick, dotted, myred] table[x index=0,y index=3]{plotdata/error2analytic_HK_d=1_potential=harmonic_multiple_epsilons.dat};
				\legend{{$\eps = 10^{-1}$},{$\eps = 10^{-2}$},{$\eps = 10^{-3}$}}
			\end{semilogyaxis}
		\end{tikzpicture}
		\vspace{-12pt}
		\caption{Evolution of the deviation between the Herman--Kluk and the analytic solution for the harmonic oscillator potential and different values of $\eps$. The error stays close to the initial sampling error even over long times. } 
		\label{fig:ErrorHKtoAnalyticHarmonic1d}
	\end{center}
\end{figure}
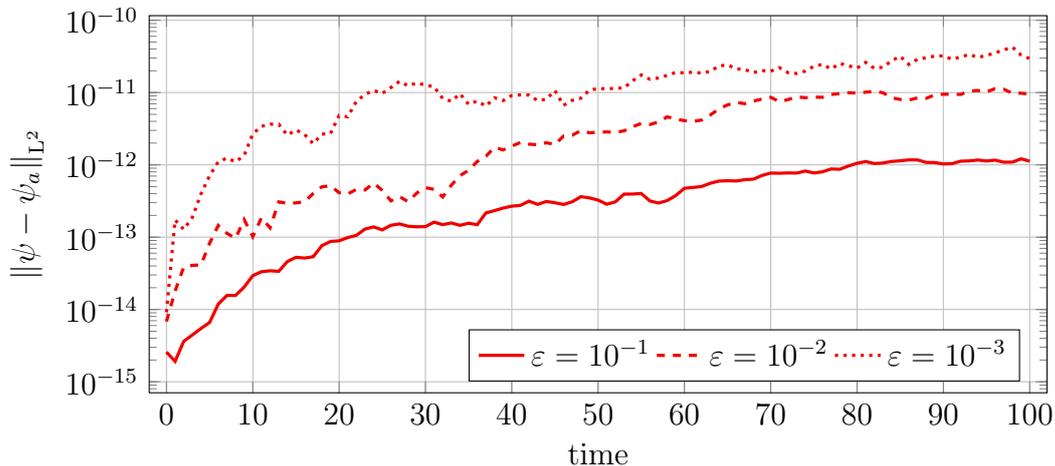

\end{example}


\begin{example}[The torsional potential in two dimensions] \label{Ex:WaveFunctionTorsionTwoD}
Intramolecular rotations are often modelled by a torsional potential of the form 
\bala
V(x) = \sum_{k=1}^d (1-\cos(x_k)). 
\eala
In two dimensions we can still evaluate the Herman--Kluk wave function on an equidistant grid and thus compare it to a reference solution that we calculated with a split-step Fourier method. As initial datum we consider a Gaussian wave packet $\psi_0 = \geps{z_0} $ with $x_0 = (1,0)^T$ and $\xi_0 = (0,0)^T$. Note that this is the same initial wave function as used in \cite[Section 5]{Faou2009}. Figure~\ref{FourierHKQ_d=2_potential=torsion_epsilon=0.01_MultipleMs_Nx=256__errors} shows the error between the reference solution and the HK wave function in the $\L^2$ norm as a function of time. The total error is a combination of the asymptotic error of order $\eps$, the quadrature error which depends on the number $M$ of quasi-Monte Carlo points, and the time discretisation error of the symplectic method of order $\gamma = 8$. We choose several different values of $M$ for both $\eps = 10^{-1}$ and $\eps = 10^{-2}$ in order to illustrate the behaviour of the error with respect to these parameters. 

\begin{figure}[!ht] 
\begin{centering}
 	 	 \begin{tikzpicture} 
 	 	 	 \begin{semilogyaxis}[ 	
			title = {$\eps=10^{-1}$}, 	
			title style={align=left, font=\normalsize},
			width = 0.33 \textwidth, 	
			height = 5cm, 		
			scale only axis, 
			xtick = {0,5,10,15,20},
			ytickten={-3,-2,-1,0},
			ymin = 0.0005,
			ymax = 1, 
			ylabel near ticks, yticklabel pos=left, 
			xlabel={time}, 
			ylabel={$ \norm{\psi - \psi_{\textrm{ref} \ } }_{L^2} $}, 
			ylabel style={font=\footnotesize}, 
			grid=major] 
			\addplot[very thick,myred,dotted] table[x index=0,y index=1]{plotdata/errors_HK_vs_Fourier_d=2_epsilon=0.1_Nx=256.dat}; 
			\addplot[very thick,myred, dashed] table[x index=0,y index=3]{plotdata/errors_HK_vs_Fourier_d=2_epsilon=0.1_Nx=256.dat}; 
			\addplot[very thick,myred] table[x index=0,y index=5]{plotdata/errors_HK_vs_Fourier_d=2_epsilon=0.1_Nx=256.dat}; 
	 	 	 \end{semilogyaxis} 
 	 	 \end{tikzpicture} 
\hfill
%
 	 	 \begin{tikzpicture} 
 	 	 	 \begin{semilogyaxis}[ 	
			title = {$\eps=10^{-2}$}, 	
			title style={align=left, font=\normalsize},
			width = 0.33 \textwidth, 	
			height = 5cm, 		
			scale only axis, 
			xtick = {0,5,10,15,20},
			ytickten={-3,-2,-1,0},
			ylabel near ticks, yticklabel pos=right, 
			ymin = 0.0005,
			ymax = 1, 
			xlabel={time}, 
			ylabel={$ \norm{\psi - \psi_{\textrm{ref} \ } }_{L^2} $}, 
			ylabel style={font=\footnotesize,anchor=near ticklabel}, 
			grid=major,
			legend columns = -1, 
			legend style={draw = none}, 
	 		legend to name= LegendQMCMCeps001, 
			legend entries={$M=2048\quad$,$M=8192\quad$,$M=32768\quad$}] 
			\addplot[very thick,myred,dotted] table[x index=0,y index=2]{plotdata/errors_HK_vs_Fourier_d=2_epsilon=0.01_Nx=256.dat}; 
			\addplot[very thick,myred,dashed] table[x index=0,y index=4]{plotdata/errors_HK_vs_Fourier_d=2_epsilon=0.01_Nx=256.dat}; 
			\addplot[very thick,myred] table[x index=0,y index=6]{plotdata/errors_HK_vs_Fourier_d=2_epsilon=0.01_Nx=256.dat}; 
	 	 	 \end{semilogyaxis} 
 	 	 \end{tikzpicture} 
\ref*{LegendQMCMCeps001}
\caption{Propagation of the error between the Herman--Kluk and the reference solution for the torsional potential in the $\L^2$ norm. On the left hand side the semi-classical parameter is chosen to be $\eps = 0.1$, on the right hand side $\eps = 0.01$. } 
\label{FourierHKQ_d=2_potential=torsion_epsilon=0.01_MultipleMs_Nx=256__errors} 
\end{centering}
\end{figure}
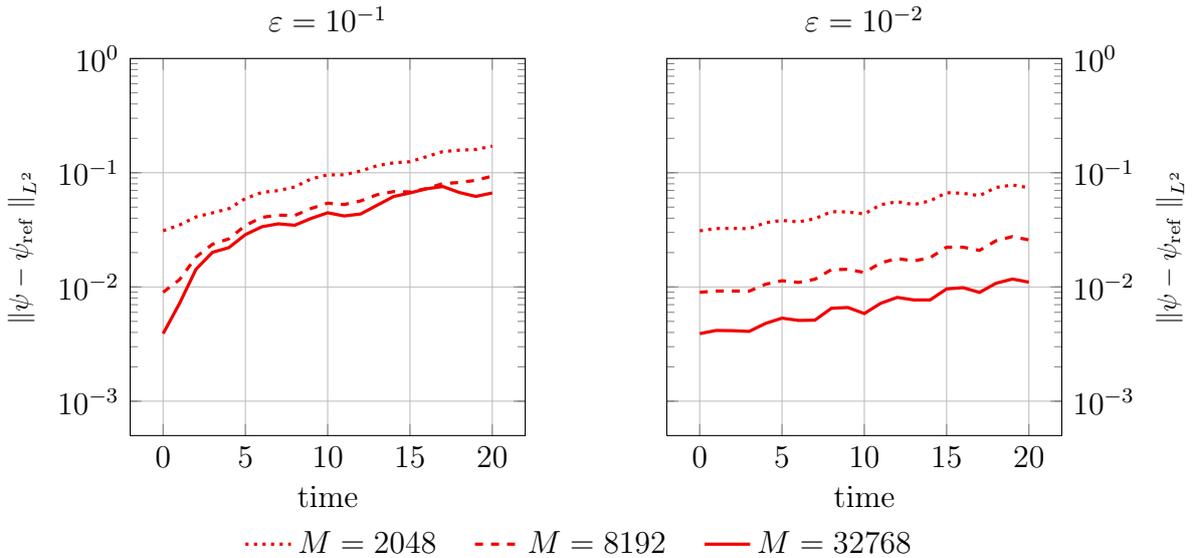 

\end{example}


\subsection{Dependence on the time step size}
As predicted by our main result, Theorem \ref{Thm:ConvergenceOfMethod}, the time discretization error of our method should behave as 
\bala
\norm{\tilde{\mathcal I^\eps_\tau} - U^\eps_\tau} \le C \left(\tau^\gamma/\eps + \tau^\gamma + \eps\right).
\eala
To underline this result by numerical calculations let us consider the same initial wave function and potential as in Example \ref{section:ExpTorsionalPotentialTwoDim}. We want to observe the behaviour for different length of time steps while the number of quasi-Monte Carlo points in phase space $M = 8192$ remains fixed. We will do so for two different values of the semi-classical parameter, namely $\eps = 10^{-1}$ and $\eps = 10^{-2}$ to show that the overall error is dominated by $\eps$ if the length of a time step becomes sufficiently small. We will use the classical St\o rmer-Verlet scheme as time integrator, i.e. $\gamma = 2$, as well as a composition method of order $\gamma = 4$. Figure \ref{Fourier_vs_HKQ_d=2_potential=torsion_M=8192_Nx=256_multiple_dts} shows the behaviour of the error 
\bala
\norm{\psi(T) - \psi_{\textrm{ref}}(T)}
\eala
between the HK solution and the reference solution at the final time $T=20$. As expected, the order of the method influences the step size at which the asymptotic 
error of order $\eps$ starts to dominate. 	
	
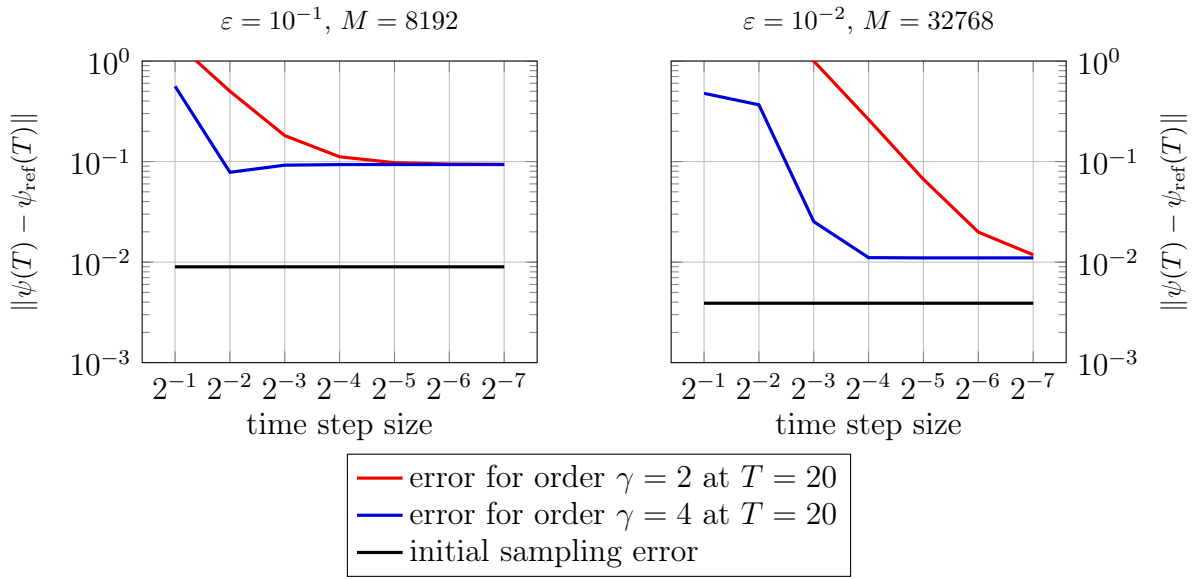
\begin{figure}[!ht] 
\centering
	\begin{tikzpicture} 
		\begin{loglogaxis}[ 
			title = {$\eps=10^{-1}$, $M=8192$}, 	
			title style={align=left, font=\footnotesize},
			scale only axis, 
			height=4cm, 
			width=0.33 \textwidth, 
			xlabel={time step size}, 
			ylabel={$ \norm{\psi(T) - \psi_{\textrm{ref}}(T)} $}, 
			ylabel style={font=\footnotesize,anchor=near ticklabel}, 
			grid=major, 
 		log basis x={2},
			x dir=reverse,
			ymax = 1, 
			ymin=0.001,
			legend columns = 1, 
			legend to name= LegendAry2, 
			legend entries={error for order $\gamma = 2$ at $T=20$, error for order $\gamma = 4$ at $T=20$, initial sampling error}, 
			legend style={cells={anchor=west}, 
			legend pos=south east}]
			\addplot[very thick,myred] table[x index=0,y index=2]{plotdata/Multiple_Stepsizes_potential=torsion_epsilon=0.1_M=8192_Nx=256_order=2.dat}; 
			\addplot[very thick,myblue] table[x index=0,y index=2]{plotdata/Multiple_Stepsizes_potential=torsion_epsilon=0.1_M=8192_Nx=256_order=4.dat}; 
			\addplot[very thick,black] table[x index=0,y index=1]{plotdata/Multiple_Stepsizes_potential=torsion_epsilon=0.1_M=8192_Nx=256_order=2.dat}; 
		\end{loglogaxis} 
	\end{tikzpicture} 
	\hfill	
	\begin{tikzpicture} 
		\begin{loglogaxis}[ 
			title = {$\eps=10^{-2}$, $M=32768$}, 	
			title style={align=left, font=\footnotesize},
			scale only axis, 
			height=4cm, 
			width=0.33 \textwidth, 
			xlabel={time step size}, 
			ylabel near ticks, yticklabel pos=right, 
			ylabel={$ \norm{\psi(T) - \psi_{\textrm{ref}}(T)} $}, 
			ylabel style={font=\footnotesize,anchor=near ticklabel}, 
			grid=major, 
 		log basis x={2},
			x dir=reverse,
			ymax = 1, 
			ymin=0.001]
			\addplot[very thick,myred] table[x index=0,y index=2]{plotdata/Multiple_Stepsizes_potential=torsion_epsilon=0.01_M=8192_Nx=256_order=2.dat}; 
			\addplot[very thick,myblue] table[x index=0,y index=2]{plotdata/Multiple_Stepsizes_potential=torsion_epsilon=0.01_M=8192_Nx=256_order=4.dat}; 
			\addplot[very thick,black] table[x index=0,y index=1]{plotdata/Multiple_Stepsizes_potential=torsion_epsilon=0.01_M=8192_Nx=256_order=2.dat}; 
		\end{loglogaxis} 
	\end{tikzpicture} 
	\ref*{LegendAry2}
	\caption{Dependence of the error between the Herman--Kluk and the reference solution on the length of one time step for the torsional potential for 
	two values of the semi-classical parameter: $\eps=0.1$ in the upper panel, $\eps=0.01$ in the lower panel.} 
	\label{Fourier_vs_HKQ_d=2_potential=torsion_M=8192_Nx=256_multiple_dts}
\end{figure} 


\subsection{Expectation values} \label{section:ExamplesExpectationValues}
For space dimensions greater than three, the computational effort to produce reference solutions with split-step Fourier or Galerkin methods is enormous. In order to show that our algorithm still produces proper results we will now calculate expectation values for higher dimensions with the Herman--Kluk propagator as described in \S\ref{section:ExpectationValues}. 


\begin{example}[The torsional potential in two dimensions] \label{section:ExpTorsionalPotentialTwoDim}
Let us consider the same setting as in Example \ref{Ex:WaveFunctionTorsionTwoD}, i.e. the torsional potential in two dimensions with Gaussian initial wave function and $\eps = 10^{-2}$. We use $ M = 8192 $ quasi-Monte Carlo points in phase space. The length of a time step is $ \tau = 0.25$ and we observe the system up to a final time $T=20$. Figure \ref{FourierHKQ_d=2_potential=torsion_epsilon=0.01_M=8100_Nx=512} shows the evolution of the energy expectation values and their respective point-wise error at every time step. The black dotted lines are the reference solution calculated by a split-step Fourier method. 
\begin{figure}[!ht] 
\centering
	\begin{tikzpicture} 
		\begin{axis}[ 
		scale only axis, 
		height=4cm, 
		width = 0.33 \textwidth,
		xlabel={time}, 
		ylabel={energy}, 
		grid=major, 
		legend columns = -1, 
		legend to name= LegendAry, 
		legend entries={kinetic energy\,, potential energy\,, total energy\,, reference}, 
		legend style={draw=none}] 
		\addplot[thick,mygreen] table[x index=0,y index=1]{plotdata/FourierHKQ_d=2_potential=torsion_epsilon=0.01_M=8192_Nx=256__energies.dat}; 
		\addplot[thick,myblue] table[x index=0,y index=2]{plotdata/FourierHKQ_d=2_potential=torsion_epsilon=0.01_M=8192_Nx=256__energies.dat}; 
		\addplot[thick,myred] table[x index=0,y index=3]{plotdata/FourierHKQ_d=2_potential=torsion_epsilon=0.01_M=8192_Nx=256__energies.dat}; 
		\addplot[thick,black,dashed] table[x index=0,y index=4]{plotdata/FourierHKQ_d=2_potential=torsion_epsilon=0.01_M=8192_Nx=256__energies.dat}; 
		\addplot[thick,black,dashed] table[x index=0,y index=5]{plotdata/FourierHKQ_d=2_potential=torsion_epsilon=0.01_M=8192_Nx=256__energies.dat}; 
		\addplot[thick,black,dashed] table[x index=0,y index=6]{plotdata/FourierHKQ_d=2_potential=torsion_epsilon=0.01_M=8192_Nx=256__energies.dat}; 
		\end{axis} 
	\end{tikzpicture} 
	\hfill
	\begin{tikzpicture} 
		\begin{semilogyaxis}[ylabel near ticks, yticklabel pos=right, 
		scale only axis, 
		height=4cm, 
		width = 0.33 \textwidth,
		xlabel={time}, 
		ylabel={relative error}, 
		grid=major] 
		\addplot[thick,mygreen] table[x index=0,y index=1]{plotdata/FourierHKQ_d=2_potential=torsion_epsilon=0.01_M=8192_Nx=256__enerrors.dat}; 
		\addplot[thick,myblue] table[x index=0,y index=2]{plotdata/FourierHKQ_d=2_potential=torsion_epsilon=0.01_M=8192_Nx=256__enerrors.dat}; 
		\addplot[thick,myred] table[x index=0,y index=3]{plotdata/FourierHKQ_d=2_potential=torsion_epsilon=0.01_M=8192_Nx=256__enerrors.dat}; 
		\end{semilogyaxis} 
	\end{tikzpicture} 
	\ref*{LegendAry}
	\caption{Evolution of energy expectation values (left) and the error to the reference solution (right) for the two-dimensional torsional potential with $\eps = 10^{-2}$.}
	\label{FourierHKQ_d=2_potential=torsion_epsilon=0.01_M=8100_Nx=512} 
\end{figure}
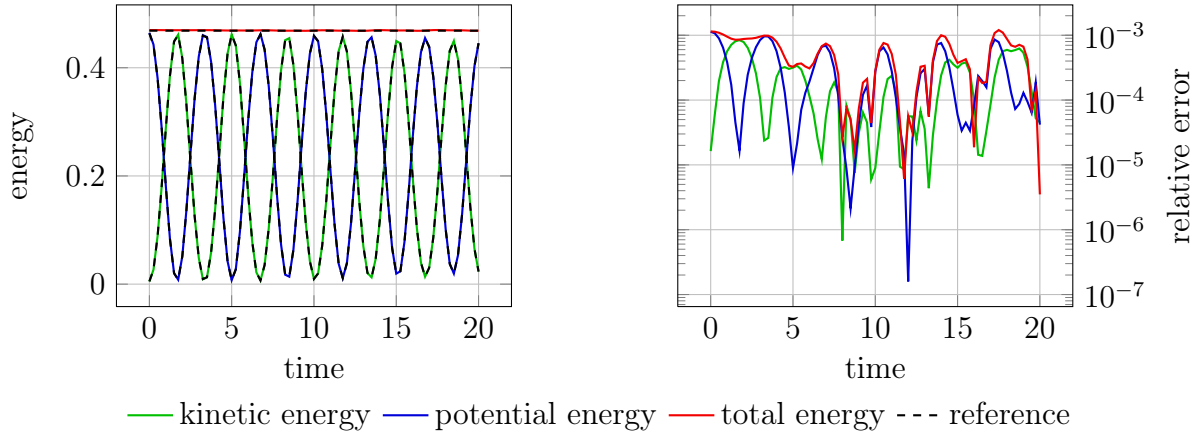

\end{example}

We conclude this section by illustrating the ability to calculate expectation values in high dimensions. 

\begin{example}[The Henon--Heiles potential] 
The Henon--Heiles potential is given by 
\bala
V(x) = \sum_{k=1}^d \frac 1 2 x_k^2 + \sum_{k=1}^{d-1} \sigma \brac{x_k x_{k+1}^2 - \frac 1 3 x_k^3} + \sum_{k=1}^{d-1} \frac{\sigma^2}{16} \brac{x_k^2 + x_{k+1}^2}^2. 
\eala
Let us consider the quantum mechanical position space to be six-dimensional which leads to a twelve-dimensional phase space. The same problem is treated in \cite[Section 5.4]{Faou2009} and \cite[Section 6]{Lasser2010} so that we may compare the results. This means that we choose the semi-classical parameter to be $\eps=10^{-2}$, the coupling constant $\sigma= 1 / \sqrt{80}$, and the initial datum as a Gaussian wave packet centred at $x_0 = (2, \dots, 2)^T$ and $\xi_0 = (0, \dots, 0)^T$. 
We use a time step size of $\tau = 0.01$. At every twentieth time step we calculate the kinetic, potential, and total energy, as well as the $L^2$ norm of our approximate solution by the method described in \S\ref{section:ExpectationValues} using $M = 4 \cdot 1024 \cdot 1024 = 2^{22}$ Halton points as quadrature nodes. Figure \ref{fig:henon6devolution} shows the evolution of kinetic, potential, and total energy. The required computation time is approximately $28$ minutes. 
As mentioned in \S\ref{sec:quasi}, expectation values can also be computed by quasiclassical approximations. We use the algorithm described in \cite{Lasser2010} as a reference solver for validating our method. Figure \ref{fig:henon6derror} shows the respective errors. 
\begin{figure}[!ht]
\begin{subfigure}[b]{\textwidth}
		\begin{tikzpicture}
			\begin{axis}[width = .85 \textwidth, 
				height = 5cm, 		
				scale only axis, 
				ytick={0,5,10,15},
				ymin = -4,
				ymax = 16.5, 
				xmin =0, 
				xmax = 20,
				xlabel={time}, 
				ylabel={energy},
				grid=major,
				legend columns = -1, 
				legend entries={kinetic, potential, total,reference},
				legend style={cells={anchor=west},
				legend pos= south west}]
				\addplot[very thick,mygreen] table[x index=0,y index=1]{plotdata/henon_6d_hk_vs_eg.dat};
				\addplot[very thick,myblue] table[x index=0,y index=2]{plotdata/henon_6d_hk_vs_eg.dat}; 
				\addplot[very thick,myred] table[x index=0,y index=3]{plotdata/henon_6d_hk_vs_eg.dat};
				\addplot[thick,black,dotted] table[x index=0,y index=5]{plotdata/henon_6d_hk_vs_eg.dat};
				\addplot[thick,black,dotted] table[x index=0,y index=6]{plotdata/henon_6d_hk_vs_eg.dat}; 
				\addplot[thick,black,dotted] table[x index=0,y index=7]{plotdata/henon_6d_hk_vs_eg.dat};
			\end{axis}
		\end{tikzpicture}
	\caption{evolution of energy expectation values} \label{fig:henon6devolution}
\end{subfigure} 
\vspace{-.5em}\\ 
\begin{subfigure}[b]{\textwidth}
		\begin{tikzpicture}
			\begin{semilogyaxis}[width = .85 \textwidth, 	
				height = 5cm, 		
				scale only axis, 
				xlabel={time},
				ylabel={error},
				ytickten={-4,-3,-2,-1},
				ymin = 0.00001,
				xmin =0, 
				xmax = 20,
				grid=major,
				legend columns = -1, 
				legend entries={kinetic, potential, total},
				legend style={cells={anchor=west},
				legend pos= south west}]
				\addplot[thick,mygreen] table[x index=0,y index=9]{plotdata/henon_6d_hk_vs_eg.dat};
				\addplot[thick,myblue] table[x index=0,y index=10]{plotdata/henon_6d_hk_vs_eg.dat}; 
				\addplot[thick,myred] table[x index=0,y index=11]{plotdata/henon_6d_hk_vs_eg.dat};
			\end{semilogyaxis}
		\end{tikzpicture}
\caption{deviation from reference} \label{fig:henon6derror}
\end{subfigure}
\caption{Time evolution of the energy expectation values and their deviation from the reference solution for the Henon--Heiles potential in 6d.} \label{fig:henon6d}
\end{figure}
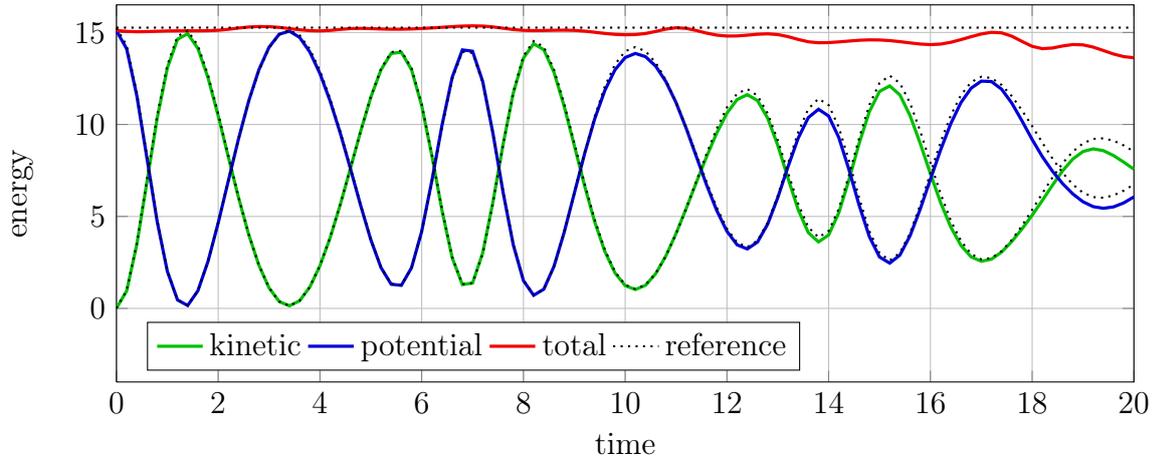
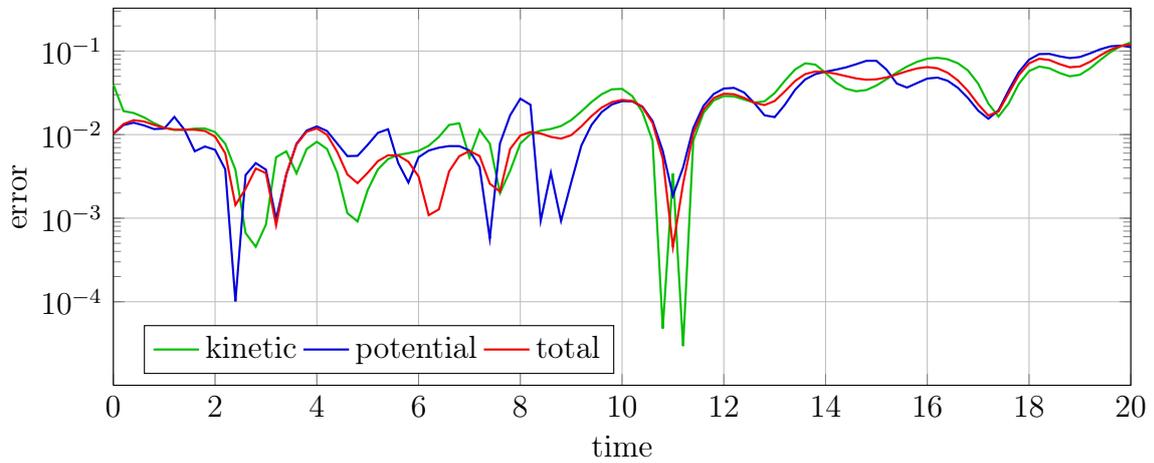
\end{example}


\appendix


\section{A detail for computing the HK factor} \label{appendix:ComputationHKFactor}
We now shall explain a method to calculate the square root that defines the Herman--Kluk factor \eqref{eqn:HKFactor}, i.e. 
\bala
u(t, z) := \sqrt{2^{-d} \det \brac{ \pd_q X^t(z) - \i \pd_p X^t(z) + \i \pd_q \Xi^t(z) + \pd_p \Xi^t(z) } }
\eala
for $t \in [0, T]$ and $z \in \R^{2d}$. We want $u$ to be continuous with respect to $t$ and therefore we need to use a continuous complex square root. In order to do so, let us introduce the notion of a continuous choice of argument for a complex-valued curve. 
\begin{definition*}
Let \( \gamma: [0, T] \to \C \smallsetminus \{0\} \) be a path. A continuous real-valued function \( h:[0, T] \to \R\) is called \emph{continuous choice of argument} along \( \gamma \) if\bala
 \gamma(t) = |\gamma(t)|e^{\ii h(t)}
\eala
holds for all \(t \in [0, T]\). 
\end{definition*}
One can prove that such a continuous choice of the argument exists. Furthermore, any two continuous choices of argument for the same path differ by a constant function and the constant must be an integer multiple of $2\pi$. This makes it possible to define a continuous complex square root by
\bala
\sqrt{\gamma(t)} := \sqrt{|\gamma(t)|} \, \expo{\frac \ii 2 \, h(t)}. 
\eala
The numerical integrator described in \S\ref{sec:TimeDiscretisation} evolves the matrices $\pd_q X^t$, $\pd_p X^t$, $\pd_q \Xi^t$, and $\pd_p \Xi^t$ in time. Additionally we calculate the absolute value and a continuous argument for 
\bala
\det \brac{ \pd_q X^t(z) - \i \pd_p X^t(z) + \i \pd_q \Xi^t(z) + \pd_p \Xi^t(z) }.
\eala 
This allows us to evaluate $u(t, z)$ whenever we need it. It also eliminates additional error sources that may arise from numerically checking the continuity of the square root. 


\section{Formulas for expectation values} \label{appendix:ExpectiationValues}
In \S\ref{section:ExpectationValues} we discuss a way to calculate an approximation to the expectation value of an observable $\Obs$. The process involves evaluating integrals of the form 
\bala
 \int_{\R^d} \overline{\geps{z^{(1)}} (x)} \, \Obs \geps{z^{(2)}} (x) \, dx 
\eala
where $z^{(1)} = (q^{(1)}, p^{(1)}) \in \R^{2d}$ and $z^{(2)} = (q^{(2)}, p^{(2)}) \in \R^{2d}$ are elements of phase space,cf. Equation \eqref{eqn:ApproxExpectationValues}. As mentioned above, there are several cases in which this integral may be computed analytically. Let us give some examples. 

\begin{example} \label{Example:ExpectationValue}
Let us first consider the case $\Obs_1 = \Id$. This means that we have to compute the scalar product of two Gaussian wave packets with the same width parameter but possibly different centres. We obtain
\bala
& \scal{\geps{z^{(1)}}}{\geps{z^{(2)}}} = \int_{\R^d} \overline{\geps{z^{(1)}} (x)} \ \geps{z^{(2)}} (x) \ dx \\ 
& = \brac{\pi \eps}^{-\frac d 2} \int_{\R^d} \exp\brac{ -\frac{1}{2 \eps} \brac{{\abs{x - q^{(1)}}^2} + \abs{x - q^{(2)}}^2}} \dots \\
& \phantom{= \brac{\pi \eps}^{-\frac d 2}} \exp\brac{ - \frac{\i}{\eps} \, p^{(1)} \cdot \brac{x - q^{(1)}} + \frac{\i}{\eps} \, p^{(2)} \cdot \brac{x - q^{(2)}} } dx \\
& = \brac{\pi \eps}^{-\frac d 2} \int_{\R^d} \exp\brac{-\frac 1 \eps \abs{x-\frac 1 2 q^{(1)} + \frac 1 2 q^{(2)} + \frac \i 2 \brac{p^{(1)} - p^{(2)}}}^2} \ dx \ \dots \\
& \phantom{= \brac{\pi \eps}^{-\frac d 2} } \exp \brac{\frac 1 {4 \eps} \abs{q^{(1)} + q^{(2)} + \i \brac{p^{(1)} - p^{(2)}}}^2 - \frac 1 {2 \eps} \brac{\abs{q^{(1)}}^2 + \abs{q^{(2)}}^2}} \dots \\
& \phantom{= \brac{\pi \eps}^{-\frac d 2} } \exp \brac{ \frac \i \eps \brac{p^{(1)} \cdot q^{(1)} - p^{(2)} \cdot q^{(2)}} } \\
& = \exp \brac{-\frac 1 {4 \eps} \abs{z^{(1)}-z^{(2)}}^2 - \frac \i {2 \eps} \brac{p^{(1)} + p^{(2)}} \cdot \brac{q^{(1)} - q^{(2)}} } . 
\eala
\end{example}

\begin{example}[Harmonic oscillator] \label{Example:ExpectationValueHarmonic}
Consider $\Obs_2 = \abs{x}^2 / \, 2$ to be the potential energy of the harmonic oscillator. By partial integration one obtains 
\bala
\scal{\geps{z^{(1)}}}{\Obs_2 \, \geps{z^{(2)}}} & = \int_{\R^d} \overline{\geps{z^{(1)}} (x)} \, \sum_{k=1}^d \frac{x_k^2}{2} \, \geps{z^{(2)}} (x) \, dx \\
& = \sum_{k=1}^d \frac 1 8 \brac{2 \eps - \brac{(p^{(1)}_k - p^{(2)}_k) + \ii \, (q^{(1)}_k + q^{(2)}_k)}^2} \scal{\geps{z^{(1)}}}{\geps{z^{(2)}}} . 
\eala
\end{example}

\begin{example}[Kinetic energy] \label{Example:ExpectationValueKinetic}
Using an $\eps$-scaled version of the Fourier transform allows us to calculate the above integral for the kinetic energy operator $\Obs_3 = -\tfrac{\eps^2}2 \Delta$. The result is 
\bala
\scal{\geps{z^{(1)}}}{\Obs_3 \, \geps{z^{(2)}}} & = \int_{\R^d} \overline{\geps{z^{(1)}} (x)} \brac{- \frac{\eps^2}2 \Delta} \geps{z^{(2)}} (x) \, dx \\
& = \sum_{k=1}^d \frac 1 8 \brac{2 \eps - \brac{(p^{(1)}_k + p^{(2)}_k) + \ii \, (q^{(1)}_k - q^{(2)}_k)}^2} \scal{\geps{z^{(1)}}}{ \geps{z^{(2)}}} . 
\eala
\end{example}

\begin{remark}
Note that, by repeated application of the techniques used in Examples \ref{Example:ExpectationValueHarmonic} and \ref{Example:ExpectationValueKinetic}, analytic expressions may be found for any observable that is a polynomial of position and momentum operator. This includes the Henon--Heiles potential. 
\end{remark}
We need not restrict ourselves to polynomial observables. The above integral may also be calculated analytically for trigonometric potentials. 
\begin{example}[Torsional potential] \label{Example:ExpectationValueTorsion}
Consider $\Obs_4$ to be the torsional potential. In a similar manner as in the previous examples one obtains the expression 
\bala
&\scal{\geps{z^{(1)}}}{\Obs_4 \, \geps{z^{(2)}}} = \int_{\R^d} \overline{\geps{z^{(1)}} (x)} \, \sum_{k=1}^d \brac{1-\cos(x_k)} \, \geps{z^{(2)}} (x) \, dx \\
& = \sum_{k=1}^d \brac{1 -\e^{- \eps / 4} \cosh\brac{\frac 1 2 \, (p^{(1)}_k - p^{(2)}_k) + \frac \ii 2 \, (q^{(1)}_k + q^{(2)}_k)}} \scal{\geps{z^{(1)}}}{ \geps{z^{(2)}}} . 
\eala
\end{example}


\section{Quasi-Monte Carlo quadrature} \label{appendix:QuasiMonteCarlo}

We will now provide the proof of Lemma~\ref{lem:Koksma} by establishing the following result that applies for even and odd dimension.
\begin{lemma}\label{theorem:KH} Let $f\in\SF(\R^d)$ and $\mu$ a probability measure on~$\R^d$ such that $f\in L^1(\d\mu)$. Then, for all $x_1,\ldots,x_M\in\R^d$, 
\bala
\lefteqn{\frac{1}{M} \sum_{m=1}^M f(x_m) - \int_{\R^d} f(x) \d\mu(x)}\\
&
= (-1)^d \int_{\R^d} \pd^{1:d} f(y) \left( \frac{1}{M}\sum_{m=1}^M \chi_{]-\infty,y]}(x_m) -\mu(]-\infty,y])\right) \d y.
\eala
\end{lemma}
Our argument adjusts the proof of the Koksma--Hlawka inequality \cite[Theorem~1]{Aistleitner2015} which holds for the integration of functions of bounded variation on the unit cube, to the integration of Schwartz functions on unbounded domains. 
\begin{proof}
For any $x\in\R^d$ we have
\bala
f(x) &= -\int_{x_1}^\infty \pd_1 f(y_1,x_2,\ldots,x_n) \d y_1 \\
&= (-1)^d \int_{x_1}^\infty \cdots \int_{x_d}^\infty \pd^{1:d} f(y_1,\ldots,y_d) \d y_d \cdots \d y_1\\ 
&= (-1)^d \int_{[x,\infty[} \pd^{1:d} f(y) \d y.
\eala
This implies for the arithmetic mean
\bala
\frac1M \sum_{m=1}^M f(x_m) &= 
\frac{(-1)^d}{M} \sum_{m=1}^M \int_{\R^d} \chi_{[x_m,\infty[}(y) \;\pd^{1:d} f(y) \d y\\
&= (-1)^d \int_{\R^d} \pd^{1:d} f(y) \frac1M\sum_{m=1}^M \chi_{]-\infty,y]}(x_m) \d y
\eala
and for the integral
\bala
\int_{\R^d} f(x) \d\mu(x) =(-1)^d \int_{\R^d} \pd^{1:d} f(y) \mu(]-\infty,y]) \d y,
\eala
where the last equation also uses Fubini's theorem. 
\end{proof}


\printbibliography


\end{document}
